\documentclass[preprint,10bp,3p,textwidth=31pc,textheight=194.25mm]{elsarticle}
\usepackage{amsmath,amsthm,amsfonts,amssymb}
\usepackage{subfig}
\usepackage{blkarray}
\usepackage{caption}
\usepackage{float}
\usepackage{algpseudocode}
\usepackage{tikz}
\usepackage{multirow}
\usepackage{hyperref}
\usepackage{circuitikz}
\usepackage{blindtext, rotating}
\usepackage{makecell}
\usepackage{amsfonts}
\hypersetup{
    colorlinks=true,
    linkcolor=blue,
    urlcolor=blue}

\numberwithin{equation}{section}
\usepackage{xcolor}
\usepackage{colortbl}
\usepackage[normalem]{ulem}
\usepackage{sectsty}
\definecolor{astral}{RGB}{46,116,181}
\subsectionfont{\color{astral}}
\sectionfont{\color{astral}}
\usepackage{colortbl}
\DeclareMathAlphabet{\mathpzc}{OT1}{pzc}{m}{it}
\DeclareFontFamily{OT1}{pzc}{}
\DeclareFontShape{OT1}{pzc}{m}{it}{<-> s * [0.900] pzcmi7t}{}
\DeclareMathAlphabet{\mathpzc}{OT1}{pzc}{m}{it}

\usepackage{amsbsy}
\usepackage{amsmath}
\usepackage{accents}
\newlength{\dhatheight}

\newenvironment{frcseries}{\fontfamily{frc}\selectfont}{}

\DeclareMathAlphabet\mathbfcal{OMS}{cmsy}{b}{n}
\definecolor{darkslategray}{rgb}{0.18, 0.31, 0.31}
\definecolor{warmblack}{rgb}{0.0, 0.26, 0.26}

\usepackage{multirow}
\usepackage{mathtools}
\usepackage{algorithm}
\usepackage{algorithmicx}
\makeatletter
\def\BState{\State\hskip-\ALG@thistlm}
\makeatother
\newtheorem{theorem}{Theorem}[section]
\newtheorem{note}{Note}[section]
\newtheorem{lemma}[theorem]{Lemma}
\newtheorem{corollary}[theorem]{Corollary}
\theoremstyle{definition}
\usepackage{hyperref}
\newtheorem{definition}{Definition}[section]
\newtheorem{remark}{Remark}[section]
\newtheorem{example}{Example}[section]

\journal{Journal of Applied Mathematics and Computing}

\begin{document}
\begin{frontmatter}
\title{\textcolor{warmblack}{\bf Direct Methods for Singular Fuzzy Linear Systems Using Generalized Inverses and Its Applications}}
\cortext[cor1]{Corresponding author}
\author[label1]{Abinaya R}
\address[label1]{Department of Mathematics, School of Advanced Sciences,  Vellore Institute of Technology, Chennai, India}
\ead{abinaya.r2024@vitstudent.ac.in}
\author[label1]{Ashish Kumar Nandi \corref{cor1}}
\ead{ashishkumar.nandi@vit.ac.in}

\begin{abstract} 
Fuzzy matrices provide an effective framework for modeling uncertainty in scientific and engineering problems, particularly fuzzy linear systems (FLS). This work transforms a general FLS into a crisp linear system using an embedding approach and reduces it to a standard block-structured form via column operations. A direct LU decomposition is developed under suitable range conditions, enabling the computation of minimum-norm solutions of rectangular FLS using the generalized $\{1,2\}$, $\{1,2,3\}$, and Moore–Penrose inverses. A full-rank decomposition is further proposed to compute the Moore–Penrose inverse for arbitrary rectangular matrices, and consistency conditions for these inverses are established. A unified framework for obtaining strong fuzzy solutions based on monotonicity and non-negativity constraints are presented. Efficient algorithms based on $\mathcal{L}U$, $QR$, and $SVD$ decompositions are developed for the block-structured matrix. The applicability and computational efficiency of the proposed methods are demonstrated through fuzzy circuit equations and Markov chain processes.
\end{abstract}
\end{frontmatter}
\section{Introduction} \label{intro}
\par
Matrices provide a compact and structured way to handle multiple variables, enabling complex calculations and solving higher-dimensional problems in various scientific and engineering domains. For non-singular linear systems, traditional methods such as $LU$ decomposition is widely used to directly compute unique solution. However, the complexity of solving these systems arises when the coefficient matrices are rectangular. This limitation has led to the development of generalized inverses, such as the reflexive generalized inverse \cite{reflexive}, inner and outer inverses \cite{Inner_outer}, and Moore-Penrose inverses \cite{Moore,Penrose}, for determining a least squares solution. These generalized inverses are used in the growing field of applicable mathematics in diverse areas, including optimization, game theory, and linear programming. Although these generalized inverses significantly broaden the scope for solving inconsistent linear systems, practical implementation remains challenging because precise parameter values are often unavailable due to uncertainty arising from external factors, incomplete information, or measurement errors. To address this, Dubois and Prade \cite{fuzzy matrix} introduced the concept of a fuzzy matrix, where each element is represented by fuzzy numbers instead of precise values. This flexibility incorporates a powerful way to handle these uncertainties in fields including game theory \cite{gametheory}, business analysis \cite{profit}, and economic modeling \cite{Economic}, where the precise prediction of gain, loss, demand, and supply can be uncertain. Moreover, by employing these fuzzy matrices, the traditional linear system leads to the formation of a fuzzy linear system (FLS) and a fully fuzzy linear system (FFLS). In an FLS \cite{friedman}, the coefficient matrix remains crisp, whereas the right-hand side vector is a fuzzy vector. Extending this idea, Dubois and Prade \cite{FFLS} initially proposed an FFLS \cite{Deh-4}, incorporating fuzziness in both the coefficient matrix and the right-hand side vector.\par
Friedman {\it et al.} \cite{friedman} first developed a general model for solving an $n \times n$ non-singular FLS using an embedding approach that replaces the original system by an equivalent $2n \times 2n$ classical linear system (crisp). In order to reduce the computational complexity of solving these systems, Abbasbandy {\it et al.} \cite{Abb-2} and Ali {\it et al.} \cite{LU-2} proposed the direct decomposition approaches, such as $LU$ and Cholesky decomposition methods, provided the coefficient matrix has a non-zero leading principal minor. Furthermore, Mohammad and Amir \cite{Inherited} proposed a type of inherited $LU$ factorization under additional conditions to refine the solving efficiency. For rectangular systems with full column rank, Matinfar {\it et al.} \cite{QR} developed a $QR$ decomposition via the Householder method. Meanwhile, Allahviranloo and Salahshour \cite{AS} introduced the $1-cut$ approach, enabling the conversion of FLS into an interval system with symmetric spreads to generate multiple solutions. The real FLS are limited when problems such as electrical circuits involve inherent imaginary components. To address this, Jahantigh {\it et al.} \cite{Complex} first explored a general method for solving a non-singular complex FLS, in which the coefficient matrix and the right-hand side vectors are converted into two real FLS and then employed Friedman {\it et al.} \cite{friedman} approach to solve these systems. Mahmoud {\it et al.} \cite{Eig} proposed an eigenvalue method using the Grobner basis which ensures exactness but requires high computational costs. To overcome such challenges, Behera and Chakraverty \cite{Beh} established two theorems for solving the $n \times n$ system, which on combining provides a final solution for the complex FLS.
Although, these techniques are effective, they become increasingly challenging to apply to high-dimensional systems.
To overcome this limitation, various iterative methods have been developed in the context of fuzzy, including Gauss-Jacobi \cite{All}, Gauss-Seidel \cite{All}, Successive Over Relaxation (SOR) \cite{Deh-1}, Adomian decomposition \cite{All-2}, Symmetric SOR \cite{Deh-1}, the conjugate gradient method \cite{Abb-1}, and Richardson \cite{Deh-1}. In parallel, generalized inverses \cite{General} offer an effective alternative to deal with the singular and non-square FLS. Asady \cite{Asady} presented a method for solving an $m \times n$ fuzzy system, considering $m \leq n$. Further, Biljana {\it et al.} \cite{MP-1} proposed a unique $\{1,3\}$ ($\{1,4\}$)-inverse to solve an FLS when the coefficient matrix is of full row (column) rank. To address inconsistency in fuzzy systems, the Moore-Penrose inverse has been widely employed to obtain the unique fuzzy least squares solutions \cite{svd,MP-1, MP-2}. Additionally, several other generalized inverses, such as the group inverse, $\{1\}$-inverse, and Drazin-inverse  have been explored in \cite{Group,1-inverse,Drazin}.\par
FLS forms the foundation for a wide range of applications involving uncertainty. They are widely used in fuzzy optimization problems \cite{Optimisation-1, Optimisation-2}, where uncertainty occurs on the right-hand side in the source of constraints or target goals, such as transportation problems, production planning, and decision-making systems. The relevance of FLS also extends to ordinary differential equations (ODEs) with fuzzy initial conditions \cite{Diff} applied in areas such as population dynamics, robotics, and modeling electrical circuits. Similarly, in the field of electrical engineering, analysis of RLC circuits \cite{Beh,Complex FLS} with uncertain resistances, voltages, capacitances, and inductances. Additionally, fuzzy modeling is increasingly applied in probabilistic areas, particularly in Markov decision processes with fuzzy rewards \cite{Markov-1,Markov-2}, which are useful in queuing models and decision analyses.\par

\subsection{Motivation}
Existing approaches for solving FLS include direct decomposition such as $LU$  \cite{Abb-1}, $QR$ \cite{QR} and Cholesky decomposition \cite{Abb-1}. While these methods are effective and widely used, they are limited to non-singular and full column rank linear systems. Motivated by these methods, this work focuses on extending such decomposition to the class of rectangular matrices. In particular, column rank revealing minimal canonical form (CRRMCF) is employed under a specific range condition, enabling the computation of $\{1,2\}-$inverse, $\{1,2,3\}-$inverse and $MP-$ inverse. In this framework, $\{1,2\}-$ inverse is employed to handle the consistent crisp system associated with the FLS, $\{1,2,3\}-$inverse addresses inconsistent FLS to obtain a least squares solution and $MP-$inverse provides the minimum norm solutions. Moreover, \cite{MP-1} provides a block structure for computing $\{1,3\}-$inverse and $\{1,4\}-$inverse which exhibit high operational counts. This limitation is addressed by the direct decomposition approaches proposed in this work which provide a computational procedure with reduced operational count. Furthermore, a general construction is developed to compute a full rank decomposition of all rectangular matrices without any specific conditions. In addition, we propose three direct decomposition algorithms to computes the $MP-$inverse. The use of these direct decomposition methods gives a structured framework, which facilitates the computation of the minimum norm least squares solution. Beyond extending the direct decomposition approaches, this work further establishes sufficient conditions for obtaining weak or strong fuzzy least squares solutions (Theorem \ref{Theorem strong 1}, \ref{Theorem strong 2}) for FLS that are not addressed in existing direct decomposition works \cite{Abb-2,LU-2}.
 \subsection{Organization}
The paper is systematically organized to provide a framework for solving the FLS using generalized inverses obtained using direct decomposition techniques. These approaches have also been validated on non-singular systems, confirming its applicability across a wider range of matrices. Section \ref{intro} presents an introduction along with a detailed literature review on FLS. Section \ref{prelim} outlines the basic terminologies, definitions, and results related to fuzzy numbers and FLS. In section \ref{main}, we developed an algorithmic approach for rectangular FLS by extending $LU$, $QR$, full rank decomposition and $SVD$ methods and employing generalized inverse including $\{1,2\}$-inverse, $\{1,2,3\}$-inverse and $MP$-inverse to obtain least squares solution. We further establish the consistency of the FLS using $\{1,2\}$-inverse,$\{1,2,3\}$-inverse and $MP-$inverse and provide the necessary and sufficient condition for a solution to be a strong fuzzy solution. Further, the operation count of the proposed method is presented and is compared with the results reported in the existing literature work. In section \ref{Numerical Example}, the validity of the proposed method is demonstrated through practical real-world problems, including analog complex RLC circuits and Markov decision processes with fuzzy rewards. 
 \section{Basic Terminology \& Preliminaries} \label{prelim}
 The key terminologies and notations used in this study are summarized below. A fuzzy set $\tilde{a}$ on the set of real numbers $\mathbb{R}$ is said to be a {\it fuzzy number} \cite{Zadeh} if it satisfies: $\tilde{a}$ must be a normal fuzzy set, the $\alpha-$cut of a fuzzy set $\tilde{a}$ must be a closed interval for every $\alpha \in (0,1]$ and support of $\tilde{a}$ must be bounded. Any arbitrary fuzzy number $\tilde{a}$ can be represented as an ordered pair of functions $(\underline{a}(\alpha),\overline{a}(\alpha))$ \cite{Goe}, which satisfies $\underline{a}(\alpha)$ is a bounded-left continuous non-decreasing function over $[0,1]$, $\overline{a}(\alpha)$ is a bounded-left continuous non-increasing function over $[0,1]$, and $\underline{a}(\alpha)\leq \overline{a}(\alpha)$, for $0 \leq \alpha \leq 1.$ The identity matrix is represented by $I$. The square zero matrix and the zero vector are denoted by $O$ and $\mathbf{O}$, respectively, while the zero matrix of appropriate dimensions is indicated by $\mathcal{O}$. An $m \times n$ real matrix $A$ with rank $\mathtt{r}$ is expressed as $A \in \mathbb{R}_{\mathtt{r}}^{m \times n}$ and the range space of $A$ is denoted as $\mathcal{R}(A).$ The lower triangular matrix is represented by $L$, while the unit upper triangular matrix is denoted by $U$. Let $S$ denote the block-structured matrix, and the scripted letter $\mathcal{L}$ refers to the Column Rank Revealing Minimal Canonical Form (CRRMCF) of $S$, which is discussed in Section \ref{main} and the obtained CRRMCF decomposition is represented by $\mathcal{L}U$. Moreover, $D^T$ denotes the transpose of matrix $D \in \mathbb{R}^{m \times n}$. For any matrix $D \in \mathbb{R}^{m \times n}$, the Moore-Penrose inverse is written as $D^{\dag}$ (also known as the $MP$-inverse), $\{1,2\}$-inverse is denoted by $D^{\{1,2\}}$ and $\{1,2,3\}$-inverse is represented by $D^{\{1,2,3\}}$.\\
Next, we provide some essential definitions related to fuzzy numbers and fuzzy matrices, which are crucial for deriving our main results.

\begin{definition}[\cite{Zimber}]
    A fuzzy number $\tilde{a}=(a,b,c,d)$ is said to be Trapezoidal Fuzzy Number (TrFN) if its membership function is defined as
    \[
    \mu_{\tilde{a}}(x)=\begin{cases}
    0 & x \geq d, x \leq a,\\
       \frac{x-a}{b-a} & a \leq x \leq b,\\
        1 & b \leq x \leq c,\\
        \frac{d-x}{d-c} & c \leq x \leq d,
    \end{cases}
    \]
    where $a \leq b \leq c \leq d.$ Here, a and d represent the lower and upper bound, respectively, while b and c define the core ( {\it i.e.,} collection of all $x$, whose membership degree is equal to 1).
\iffalse
\begin{figure}[H]
\centering
\begin{tikzpicture}[scale=2]
\draw[->,blue] (0,-0.4) -- (4.5,-0.4) node[right] {$x$};
\draw[->,blue] (0,-0.4) -- (0,1.5) node[above] {$\mu_{\tilde{A}}(x)$};
\draw[dotted] (2.0,1) -- (2,-0.4);
\draw[dotted] (2.7,1) -- (2.7,-0.4);
\node[below] at (1,-0.4) {$a$};
\node[below] at (2,-0.4) {$b$};
\node[below] at (2.7,-0.4) {$c$};
\node[below] at (3.7,-0.4) {$d$};
\node at (-0.15,-0.46) {0};
\node[left] at (0,1) {1};
\draw[thick,red] (1,-0.4) -- (2,1) --(2.7,1)-- (3.7,-0.4);
\draw[thick] (0,1) -- (2,1);
\end{tikzpicture}
\caption{TrFN Number}
\end{figure}
\fi
The $\alpha-$ cut of a TrFN $\tilde{a}$ is given by \[\tilde{a}_{\alpha}=[(b-a)\alpha+a,(c-d)\alpha+d], 0 \leq \alpha \leq 1.\] If $b=c,$ then the TrFN reduces to a Triangular Fuzzy Number(TFN).   
\end{definition}
\begin{definition} [\cite{friedman}]
An $m \times n$ linear system  
\begin{align} \label{fls}
d_{11}\tilde{z}_1+d_{12}\tilde{z}_2+\dots+d_{1n}\tilde{z}_n=\tilde{b}_1 \nonumber\\
d_{21}\tilde{z}_1+d_{22}\tilde{z}_2+\dots+d_{2n}\tilde{z}_n=\tilde{b}_2 \nonumber\\ \vdots \nonumber\\ d_{m1}\tilde{z}_1+d_{m2}\tilde{z}_2+\dots+d_{mn}\tilde{z}_n=\tilde{b}_m 
\end{align}
is said to be an $m \times n $ {\it fuzzy linear system} (FLS) if the coefficient matrix $A=(d_{ij}), 1\leq i\leq m,1\leq j\leq n$, is a real matrix (crisp matrix), and  $\tilde{\mathbf{z}}=(\tilde{z}_{j})$ and $\tilde{\mathbf{b}}=(\tilde{b_{i}})$  are fuzzy vectors, where fuzzy vector is a vector in which each element is a fuzzy number.
\end{definition}

\begin{definition}[\cite{friedman}] \normalfont 
    A fuzzy vector $(\tilde{z}_{1},\tilde{z}_{2},...,\tilde{z}_{n})^{T}$ denoted by $\tilde{z}_{j}=(\underline{z}_{j} (\alpha),\overline{z}_{j} (\alpha))$, $1 \leq j \leq n, 0 \leq \alpha \leq 1$ is called a solution of the fuzzy system if it satisfies
    \begin{eqnarray*}
    \sum_{j=1}^{n} d_{ij}\underline{z}_{j}= \underline{b}_{i},~~\sum_{j=1}^{n} d_{ij}  \overline{z}_{j}= \overline{b}_{i}, i=1,2,\dots,m. \end{eqnarray*}
\end{definition}

\begin{definition}[\cite{friedman}]\normalfont \label{embedding}
We determine a $2m \times 2n$ linear system (crisp linear system) of eqn. (\ref{fls}) as follows:
\begin{eqnarray} \label{SZ=B}
  SZ=B,  
\end{eqnarray} where
$s_{i,j}=s_{i+m,j+n}=\begin{cases}
d_{i,j} & d_{i,j} \geq 0\\
0 & \text{otherwise}
\end{cases},~
s_{i,j+n}=s_{i+m,j}=\begin{cases}
-d_{i,j} & d_{i,j} < 0\\
0 & \text{otherwise},
\end{cases}$\\
$Z=(\underline{z}_{1},...,\underline{z}_{n},-\overline{z}_{1},...,-\overline{z}_{n})^{T} \text{~and~} B=(\underline{b}_{1},...,\underline{b}_{m},-\overline{b}_{1},...,-\overline{b}_{m})^{T}, 1 \leq i \leq m, 1 \leq j \leq n$.\\ $S \in \mathbb{R}_{\mathtt{r}}^{2m \times 2n}$ can be equivalently written as \begin{eqnarray} \label{S}
 S=\begin{bmatrix} S_{1} & S_{2} \\ S_{2} & S_{1} \end{bmatrix}, \end{eqnarray}where $S_1,S_2 \in \mathbb{R}^{m \times n}$ and $A=S_{1}-S_{2}.$
\end{definition}
\begin{definition} [\cite{General}] \label{weak and strong}
Let $Z=\{ (\underline{z}_i(\alpha),\overline{z}_i(\alpha)), 1 \leq i \leq n\}$ be a solution of eqn. (\ref{SZ=B}). Define a fuzzy vector $W=\{(\underline{w}_i(\alpha),\overline{w}_i(\alpha)), 1 \leq i \leq n$\}, where
\begin{eqnarray*}
  \underline{w}_i(\alpha)&=&\min \{\underline{z}_i(\alpha),\overline{z}_i(\alpha),\underline{z}_i(1),\overline{z}_i(1)\}  \\
  \overline{w}_i(\alpha)&=&\max \{\underline{z}_i(\alpha),\overline{z}_i(\alpha),\underline{z}_i(1),\overline{z}_i(1)\}.
\end{eqnarray*}
is called a fuzzy solution of $SZ=B$. If $(\underline{z}_i(\alpha),\overline{z}_i(\alpha)), 1 \leq i \leq n$ are all fuzzy number then $\underline{w}_i(\alpha)=\underline{z}_i(\alpha)$ and $\overline{w}_i(\alpha)=\overline{z}_i(\alpha), 1 \leq i \leq n$, and we say that $W$ is a {\it strong fuzzy solution}. Otherwise, $W$ is a {\it weak fuzzy solution}. Note that the weak fuzzy solution may not be a fuzzy vector in general.
\end{definition}

The definitions of $\{1,2\}$-inverse, $\{1,2,3\}$-inverse and $MP$-inverse \cite{Ben,Moore,Penrose} corresponding to the block-structure matrix $S$ of (\ref{S}) are given as follows.
\begin{definition} \normalfont \label{{1,2}-inverse}
  Let $S \in \mathbb{R}^{2m\times 2n}$, a matrix $Y \in \mathbb{R}^{2n\times 2m}$ that satisfies the conditions $SYS=S$ and $YSY=Y$ is called a $\{1,2\}$-inverse or reflexive generalized inverse of $S$. 
\end{definition}

\begin{definition} \normalfont \label{{1,2,3}-inverse}
  Suppose $S \in \mathbb{R}^{2m\times 2n}$, a matrix $Y \in \mathbb{R}^{2n\times 2m}$ that satisfies $SYS=S$, $YSY=Y$ and $(SY)^T=SY$ is called a $\{1,2,3\}$-inverse of $S$. 
\end{definition}

\begin{definition} \normalfont  \label{MP-inverse}
    For any $S \in \mathbb{R}^{2m\times 2n}$, there exists a unique matrix $Y \in \mathbb{R}^{2n\times 2m}$ satisfying the matrix equations: $SYS=S$, $YSY=Y$, $(SY)^T=SY$ and $(YS)^T=YS$ is known as the {\it Moore-Penrose inverse} of $S$.
\end{definition}
\iffalse
\begin{definition}
Consider the partitioned matrix $A=\begin{bmatrix}
 B & C \\D & E  
\end{bmatrix} \in \mathbb{R}^{(p+r) \times (q+s)},$ where $B \in \mathbb{R}^{p \times q},C \in \mathbb{R}^{p \times s},D \in \mathbb{R}^{r \times s}$ and $E \in \mathbb{R}^{r \times s}$. The pseudo-Schur complement of $E$ in $A,$ is defined by \[A/E=B-CE^{\dag}D, \] where $E^{\dag}$ is the $MP-$inverse of $E.$
\end{definition}
\fi
Next, we present some preliminary theorems that are essential for deriving the main results.
\begin{theorem}[\cite{Cline}] \normalfont \label{Theorem 3.8}
Let $W, V \in \mathbb{R}^{m \times n}$ such that $WV^{T}=O.$ Then the $MP$-inverse of $W+V$ is given by $(W+V)^{\dag}=W^{\dag}+(I-W^{\dag}V)[M^{\dag}+(I-M^{\dag}M)NV^{T}(W^{\dag})^{T}W^{\dag}(I-VM^{\dag})],$
    where $ M=(I-WW^{\dag})V,~N=[I+(I-M^{\dag}M)V^{T}(W^{\dag})^{T}W^{\dag}V(I-M^{\dag}M)]^{-1}.$   
\end{theorem}
 
\begin{theorem}[\cite{Hung}] \normalfont  \label{Lemma 3.7}
 Consider the partitioned matrix $C=\begin{bmatrix}
D & \mathcal{O} \\ E & F \\    
\end{bmatrix},~K=D^TD+E^TE.$ Then the $MP$-inverse of $C$ is
$C^{\dag}=\begin{bmatrix}
    K^{\dag}D^T & K^{\dag}E^{T} \\ \mathcal{O} & F^{\dag}
\end{bmatrix},$ if and only if $E^{T}F=O.$
\end{theorem}

\begin{theorem}[\cite{Hung}] \normalfont  \label{Lemma 3.6}
 If the block matrix $C$ is as described in Theorem \ref{Lemma 3.7}, then 
$C^{\dag}=\begin{bmatrix}
    D^{\dag} & E^{T}K^{\dag} \\ \mathcal{O} & F^{T}K^{\dag}
\end{bmatrix},$
where $K=EE^T+FF^T $~iff~$ DE^T=O.$
\end{theorem}

\begin{theorem}[\cite{Reverse}] \normalfont
If $\mathcal{R}(A^TAB) \subseteq \mathcal{R}(B)$ and $\mathcal{R}(BB^TA^T) \subseteq \mathcal{R}(A^T),$ then $(AB)^{\dag}=B^{\dag}A^{\dag}.$
\end{theorem}
\section{Methodology} \label{main}
 This section introduces a systematic approach for computing the $\{1,2\}$-inverse, $\{1,2,3\}$-inverse and $MP$-inverse of a $2m\times 2n$ crisp system (\ref{SZ=B}). To compute these generalized inverses, we constructed a decomposition with the help of column operations that transform the matrix into a particular structured form. Further, this procedure is implemented for $LU$ decomposition to find the solution of the rectangular FLS (\ref{fls}). In addition to this, a construction to compute full rank decomposition is provided to obtain $MP-$inverse without any special conditions, making it applicable to all rectangular matrices. Following this, $QR$ and $SVD$ decompositions are derived using suitable orthogonal matrices, and these reduced forms are utilized to compute generalized inverses for solving any type of fuzzy linear systems. Moreover, we have examined the necessary and sufficient conditions for ensuring the consistency of fuzzy linear system. Finally, we also examined conditions for obtaining a strong fuzzy solution.

 \subsection{$\mathcal{L}U$ decomposition for the block-structure matrix}
 This method is elaborated from Dopico {\it et al.} \cite{Dop}, with a significant modification wherein the column operations are replaced in place of the row operations, resulting in a canonical form obtained by the right multiplication of a unit upper triangular matrix with the corresponding given matrix. %By using the elementary column operations, the span of the original columns are preserved. This helps one to verify the required range condition in Lemma \ref{L_U} using the CRRMCF of the corresponding matrix obtained.

\begin{definition} \label{CRRMCF}
           For any matrix $G \in \mathbb{R}^{m\times n}$ with $rank(G)=\mathtt{t}$, let the set $\{ (r_k,c_k)\}^{\mathtt{t}}_{k=1}$, where $c_1 <c_2< \dots <c_{\mathtt{t}},$ indicates the positions of the topmost non-zero leading entry (pivot entry) in that column. Then, the matrix $G$ is said to be in {\it Column Rank Revealing Minimal Canonical Form (CRRMCF)} if all the entries to the right of the pivot entries are zero.\\
           If a matrix $G$ is in CRRMCF, then
     \begin{enumerate}
         \item[(i)] $G(r_k,c_k) \neq $ 0 for $k \in \{1,2,\dots,\mathtt{t}\}.$
         \item[(ii)] $G(i,j) = 0 $ if $j \notin \{c_1,c_2,\dots,c_{\mathtt{t}}\},$ for $i\in\{1,2,\dots,m\}.$
         \item[(iii)] $G(i,c_k)=0$ if $i<r_k$ for $k \in \{1,2,\dots,\mathtt{t}\}.$
         \item[(iv)] $G(r_k,j)=0$ if $j>c_k$ for $k \in \{1,2,\dots,\mathtt{t}\}.$
     \end{enumerate}
\end{definition}
\begin{note} \normalfont
    Here, $r_k$ and $c_k$ denote the row and column indices, respectively of the $k-$th leading entry (pivot entry), and not the $k-$th row or the $k-$th column of the given matrix.
\end{note}
In the following lemma, we construct a CRRMCF of the block structured matrix $S \in \mathbb{R}^{2m \times 2n}.$ Instead of working with $2m \times 2n$ matrix, we focus on two $n \times n$ matrices to obtain the CRRMCF decomposition of $S$ with a specific range condition. This provides a simplified representation of the block-structure matrix $S$, and the lemma under a minimal range condition assumptions ensures that the off diagonal matrices can be incorporated without affecting the structure of $S.$ 
\begin{lemma} \normalfont \label{L_U}
For any $A \in \mathbb{R}^{m \times n}$, suppose $S\in \mathbb{R}^{2m \times 2n}$ is a block-structure matrix defined in eqn (\ref{S}). Assume that $\mathcal{L}_{11}U_{11}$ is a CRRMCF decomposition of $S_1$,  $\mathcal{L}_{22}U_{22}$ is a CRRMCF decomposition of the matrix $S_1-S_2U_{11}^{-1}\mathcal{L}_{11}^{\dag} S_2$, where $\mathcal{L}_{11},\mathcal{L}_{22} \in \mathbb{R}^{m \times n}$ are CRRMCF and $U_{11}, U_{22}\in \mathbb{R}^{n \times n}$ are unit upper triangular matrices, and \[\mathcal{R}(S_2) \subseteq \mathcal{R}(S_1).\] Then the CRRMCF decomposition of $S$ is given as \[S=\mathcal{L}U,\] where $\mathcal{L}=\begin{bmatrix}
\mathcal{L}_{11}& O\\ \mathcal{L}_{21} & \mathcal{L}_{22}\\
\end{bmatrix}$ and $
U=\begin{bmatrix}
    U_{11} & U_{12} \\ O & U_{22}\\
\end{bmatrix}$ with off-diagonal block matrices \[\mathcal{L}_{21}=S_2U_{11}^{-1}, U_{12}=\mathcal{L}_{11}^{\dag} S_2.\]
\begin{proof}
Given, $\mathcal{L}_{11}U_{11}$ be a CRRMCF decomposition of $S_1$. Now, consider the block decomposition \[\begin{bmatrix}
    S_1 & S_2 \\ S_2 & S_1
\end{bmatrix}=\begin{bmatrix}
\mathcal{L}_{11}& O\\ \mathcal{L}_{21} & \mathcal{L}_{22}\\
\end{bmatrix} 
\begin{bmatrix}
    U_{11} & U_{12} \\ O & U_{22}\\
\end{bmatrix}.\]
By comparing the corresponding blocks, we get
\begin{eqnarray}
    S_1 = \mathcal{L}_{11}U_{11}, S_2 = \mathcal{L}_{21}U_{11},
    S_2 = \mathcal{L}_{11}U_{12} \text{ and }
    S_1 = \mathcal{L}_{21}U_{12}+\mathcal{L}_{22}U_{22}. \label{LU_1}
    \end{eqnarray}
    Since $U_{11}$ is non-singular matrix and from eqn. (\ref{LU_1}) it follows that \[\mathcal{R}(\mathcal{L}_{11})=\mathcal{R}(S_1).\]
  Given $\mathcal{R}(S_2) \subseteq \mathcal{R}(S_1)$, we obtain \[\mathcal{L}_{11}\mathcal{L}_{11}^{\dag}S_2=S_2.\]
   Substituting in eqn. (\ref{LU_1}) we obtain, $S_1-S_2U_{11}^{-1}\mathcal{L}_{11}^{\dag} S_2=\mathcal{L}_{22}U_{22}.$ Since $\mathcal{L}_{11}$ and $\mathcal{L}_{22}$ are CRRMCF, the presence of $\mathcal{L}_{21}$ does not change any column structure. Thus, the obtained structure $\mathcal{L} \in \mathbb{R}^{2m \times 2n}$ is a CRRMCF of $S$ and $U \in \mathbb{R}^{2n \times 2n}$ is a unit upper triangular matrix. Therefore, the decomposition $S=\mathcal{L}U$ is a CRRMCF decomposition of $S.$
   \end{proof}
\end{lemma}

\begin{remark}
\begin{enumerate}
    \item[(i)] In Lemma \ref{L_U}, the CRRMCF decomposition of $S_1$ and $S_1-S_2{U}_{11}^{-1}\mathcal{L}_{11}^{\dag} S_2$ can be obtained using Algorithm \ref{algo1}. 
    \item[(ii)] Further, this representation is not unique, but the corresponding $\mathcal{L}$ obtained in this method is unique (Since $\mathcal{L}_{11}$ and $ \mathcal{L}_{22}$ are unique).
\end{enumerate}
\end{remark}
\begin{remark} \normalfont
Consider an $2n \times 2n$ block structured matrix $S=\begin{bmatrix} S_{1} & S_{2} \\ S_{2} & S_{1} \end{bmatrix},$ where $S_1,S_2 \in \mathbb{R}^{n \times n}$. Suppose $S_1$ is non-singular, then \[\mathcal{R}(S_1)=\mathbb{R}^n.\] Thus, the condition $\mathcal{R}(S_2) \subseteq \mathcal{R}(S_1)$ holds automatically. Hence, the CRRMCF decomposition of $S=\mathcal{L}U$ can be obtained directly without any condition.
\end{remark}
\begin{remark} \normalfont
    If $S_2=\mathcal{O},$ then $\mathcal{R}(S_2) \subseteq \mathcal{R}(S_1)$ is trivial.
\end{remark}
 The next result provides an exact $LU$ form of the matrix $S$ under some equivalent conditions.
\begin{lemma}\normalfont \label{Theorem 3.5}
    Assume that $A \in \mathbb{R}^{m\times n}$ and $S$ is a block-structure matrix of the form (\ref{S}) with $\mathcal{R}(S_2) \subseteq \mathcal{R}(S_1)$. Let $S=\mathcal{L}U $
    be a CRRMCF decomposition of $S\in \mathbb{R}_{\mathtt{r}}^{2m\times 2n}$, and let $\{(r_k,c_k)\}^{\mathtt{r}}_{k=1}$ be the set of all indices that denote the position of the pivot element in the nonzero column of $\mathcal{L}$, then the following are equivalent.
\begin{enumerate}
    \item[(i)] $r_k\geq c_k,~\mbox{for all}~ k=1,2,\hdots,{\mathtt{r}}.$ 
    \item[(ii)] $\mathcal{L}$ is lower triangular.
    \item[(iii)] $\mathcal{L}U$ is an $LU$ decomposition of $S$.
\end{enumerate}
\begin{proof}
$(i) \implies (ii) $ Suppose $r_k\geq c_k,~\mbox{for all}~ k,$ then by Definition \ref{CRRMCF} (iii), $\mathcal{L}(i,c_k)=0$ for $i<r_k$ and by Definition \ref{CRRMCF} (ii) for $i<c_k,$ non-pivot columns are zero. Thus $\mathcal{L}$ is lower triangular.\\
$(ii) \implies (i) $ If for some $k,~r_k<c_k$ then $\mathcal{L}(r_k,c_k) \neq 0,$ which contradicts $\mathcal{L}$ is lower triangular.\\
$(ii) \iff (iii) $ Proof follows directly from $LU$ decomposition \cite{Matrix computation}.
\end{proof}
\end{lemma}
In Lemma \ref{L_U}, the obtained $\mathcal{L}$ is not necessarily lower triangular. To overcome this, the following lemma ensures the required lower triangular form.

\begin{lemma} \normalfont \label{Lemma LPU}
    Let $A \in \mathbb{R}^{m \times n}$ and $S=\mathcal{L}U$ be a CRRMCF decomposition of the corresponding block-structure matrix $S$ with $\mathcal{R}(S_2) \subseteq \mathcal{R}(S_1)$. Then there exists a permutation matrix $P$ such that $$S=LP^{-1}U=\begin{bmatrix}
    L_{11} & O \\ L_{21} & L_{22}\\
\end{bmatrix}P^{-1}
\begin{bmatrix}
    U_{11} & U_{12} \\ O & U_{22}\\
\end{bmatrix} \label{S=LPU},$$ where $L=\mathcal{L}P$ is a lower triangular matrix.
\begin{proof}
    Since $S=\mathcal{L}U,$ where $U$ is a unit upper triangular matrix, we have $rank(S)=rank(\mathcal{L}),$ where ${\mathtt{r}}$ is the rank of the matrix $S.$ By Definition \ref{CRRMCF} (ii) and (iv), the non-pivot column of $\mathcal{L}$ is zero and no two pivot share the same row in $\mathcal{L}$. Thus the $\mathtt{r}$ pivot row indices $r_1,r_2,\dots, r_{\mathtt{r}}$ remains distinct. Also from Definition \ref{CRRMCF}, we have $c_1<c_2< \hdots<c_{\mathtt{r}}$. Let $P$ be a permutation matrix which rearrange the $\mathtt{r}$ pivot columns of $\mathcal{L}$ to position $1,2,\hdots,\mathtt{r}$ and thus their corresponding rows $r_1,r_2,\dots,r_{\mathtt{r}}$ appear in increasing order and the zero columns are shifted to the right. Hence after rearranging, the $i^{th}$ pivot occupies column $i$ and row index equal to the $i^{th}$ smallest of ${\mathtt{r}}$ distinct values, which implies $r_i \geq c_i$ for all $i=1,2,\hdots,\mathtt{r}.$ By Lemma \ref{Theorem 3.5}, we obtain the given result.
\end{proof}
\end{lemma}
The $\{1,2\}$-inverse is employed to handle the FLS when the associated crisp linear system is consistent. The following theorem provides the methodology to compute the $\{1,2\}$-inverse based on the decomposition obtained above. %Since $S=LP^{-1}U$, where $P$ is a permutation matrix and $U$ is a unit upper triangular matrix, we have $rank(S)=rank(L)$, where $r$ is the rank of $S.$
\begin{theorem} \normalfont \label{Theorem 3.1}
   Suppose $A \in {\mathbb{R}}^{m\times n}$ and $S=LP^{-1}U$, obtained from Lemma \ref{L_U} and \ref{Lemma LPU}. If $Q$ is a permutation matrix such that
$QL=\begin{bmatrix}
    L_{\mathtt{r}} & \mathcal{O} \\ K & \mathcal{O}
\end{bmatrix},$ then
    $S^{\{1,2\}}=U^{-1}PL^{\{1,2\}},$ where $L_{\mathtt{r}} \in \mathbb{R}_{\mathtt{r}}^{r\times r}$ and $K \in \mathbb{R}^{(m-r)\times r}$.
    In this case,
$L^{\{1,2\}}=\begin{bmatrix}
      {L_{\mathtt{r}}}^{-1} & \mathcal{O} \\ \mathcal{O} & \mathcal{O}\\
    \end{bmatrix}Q \text{ and }
    U^{-1}=\begin{bmatrix}
    U_{11}^{-1} & -U_{11}^{-1}U_{12}U_{22}^{-1} \\ O & U_{22}^{-1}\\
\end{bmatrix}$.
    \begin{proof} \normalfont 
    By Definition \ref{CRRMCF} (ii), the non-pivot columns of $\mathcal{L}$ is entirely zero and $\mathcal{L}$ has exactly ${\mathtt{r}}$ pivot columns. Let $Q$ be a permutation matrix that permute row so that ${\mathtt{r}}$ pivot rows appear first. Then the first ${\mathtt{r}}$ rows forms $L_{\mathtt{r}}$ which is nonsingular and hence we obtain $QL=\begin{bmatrix}
    L_{\mathtt{r}} & \mathcal{O} \\ K & \mathcal{O}
\end{bmatrix}.$
We first show that $L^{\{1,2\}}$ is a $\{1,2\}$-inverse of $L$, and then we proceed to prove that $S^{\{1,2\}}$ is a $\{1,2\}$-inverse of $S$. Now \begin{eqnarray}
LL^{\{1,2\}}L&=& \begin{pmatrix}
    Q^{-1}\begin{bmatrix}
    L_{\mathtt{r}} & \mathcal{O} \\ F & \mathcal{O}
\end{bmatrix}
\end{pmatrix}
\begin{pmatrix}
\begin{bmatrix}
      L_{\mathtt{r}}^{-1} & \mathcal{O} \\ \mathcal{O} & \mathcal{O}\\
    \end{bmatrix}Q
    \end{pmatrix}
    \begin{pmatrix}
    Q^{-1}
\begin{bmatrix}
    L_{\mathtt{r}} & \mathcal{O} \\ F & \mathcal{O}
\end{bmatrix}
\end{pmatrix}\nonumber\\
   &=&
   Q^{-1}\begin{bmatrix}
    L_{\mathtt{r}} & \mathcal{O} \\ F & \mathcal{O}
\end{bmatrix}
    \begin{bmatrix}
    I_{\mathtt{r}} & \mathcal{O} \\ \mathcal{O} & \mathcal{O}
\end{bmatrix}=
    Q^{-1}\begin{bmatrix}
    L_{\mathtt{r}} & \mathcal{O} \\ F & \mathcal{O}
\end{bmatrix}=L. \label{3.1}
\end{eqnarray}
Similarly, we obtain
\begin{eqnarray}
L^{\{1,2\}}LL^{\{1,2\}}
    &=&
\begin{bmatrix}
      I_{\mathtt{r}} & \mathcal{O} \\ \mathcal{O} & \mathcal{O}\\
    \end{bmatrix}
\begin{bmatrix}
      L_{\mathtt{r}}^{-1} & \mathcal{O} \\ \mathcal{O} & \mathcal{O}\\
    \end{bmatrix}Q= \begin{bmatrix}
      L_{\mathtt{r}}^{-1} & \mathcal{O} \\ \mathcal{O} & \mathcal{O}\\
    \end{bmatrix}Q=L^{\{1,2\}}. \label{3.2}
\end{eqnarray}
From eqn. (\ref{3.1}) and (\ref{3.2}), we can conclude that $L^{\{1,2\}}$ is a ${\{1,2\}}$-inverse of $L$.\\
For the matrix $U=\begin{bmatrix}
    U_{11} & U_{12} \\ O & U_{22}\\
\end{bmatrix}$, it is immediate that $UU^{-1}=I=U^{-1}U,$
where $U^{-1}=\begin{bmatrix}
    U_{11}^{-1} & -U_{11}^{-1}U_{12}U_{22}^{-1} \\ O & U_{22}^{-1} \end{bmatrix}$. Moreover, to establish the main claim that requires examining the conditions $SS^{\{1,2\}}S=S$ and $S^{\{1,2\}}SS^{\{1,2\}}=S^{\{1,2\}}.$ Consider, \begin{eqnarray}
SS^{\{1,2\}}S
&=& LP^{-1}U(U^{-1}PL^{\{1,2\}})LP^{-1}U =LP^{-1}(UU^{-1})PL^{\{1,2\}}LP^{-1}U \nonumber\\
&=&L(P^{-1}P)L^{\{1,2\}}LP^{-1}U=(LL^{\{1,2\}}L)P^{-1}U=LP^{-1}U=S.\nonumber
\end{eqnarray}
Similarly, it follows that $S^{\{1,2\}}SS^{\{1,2\}}=S^{\{1,2\}}.$ Since Hermite normal form exists for all matrices \cite{Ben}, which implies the existence of $\{1,2\}-$inverse.
\end{proof}
\end{theorem}
\begin{corollary} \normalfont \label{Corollary 3.2}
    If $S$ is non-singular, then
    \begin{eqnarray*}
    S^{\{1,2\}}=U^{-1}PL^{-1}=
     \begin{bmatrix}
    U_{11}^{-1} & -U_{11}^{-1}U_{12}U_{22}^{-1} \\ O & U_{22}^{-1}
\end{bmatrix}P\begin{bmatrix}
      L_{11}^{-1} & O \\ -L_{22}^{-1}L_{21}L_{11}^{-1} & L_{22}^{-1}\\
    \end{bmatrix}.   
\end{eqnarray*}
\end{corollary}
\iffalse
\begin{corollary} \normalfont \label{Corollary 3.3}
If $\mathcal{L}_{22}=O$ in Lemma \ref{L_U}, then 
\[
S^{\{1,2\}}=\begin{bmatrix}
 U_{11}^{-1}\mathcal{L}_{11}^{-1} & O \\ O & O   
\end{bmatrix}.
\]
\end{corollary} 
\fi
The following result characterizes the general form of $\{1,2\}-$inverse associated with the block structured matrix $S.$
\begin{theorem} \normalfont
Consider $A \in {\mathbb{R}}^{m\times n}$. Let $S=LP^{-1}U$ obtained from Lemma \ref{L_U} and \ref{Lemma LPU} and $Q$ be a permutation matrix such that
$QL=\begin{bmatrix}
    L_{\mathtt{r}} & \mathcal{O} \\ K & \mathcal{O}
\end{bmatrix},$ then the $\{1,2\}-$ inverse of $S$ is given as \[ 
    S^{\{1,2\}}=U^{-1}PE^{\{1,2\}}Q,
    \]
    where $E^{\{1,2\}}=\begin{bmatrix}
        L_{\mathtt{r}}^{-1} & \mathcal{O} \\ X_{2} L_{\mathtt{r}}^{-1} & \mathcal{O}
    \end{bmatrix}$ and $X_2$ is arbitrary.
    \begin{proof}
        Clearly, $(\mathcal{L}U)^{\{1,2\}}=U^{-1} \mathcal{L}^{\{1,2\}}.$ Also, from full rank decomposition, we have $$E\left\{1,2\right\}=\left\{\begin{bmatrix}
     I_{\mathtt{r}} \\ X_2   
    \end{bmatrix}\begin{bmatrix}
     L_{\mathtt{r}}^{-1} & \mathcal{O} \end{bmatrix} \text{ : $X_2$ is arbitrary} \right\},$$ where $E\left\{1,2\right\}$ represents the collection of all $\{1,2\}-$inverse of the matrix $E.$
     Now consider, \begin{eqnarray*}
         SS^{\{1,2\}}S&=&\left(Q\begin{bmatrix}
         L_{\mathtt{r}} & \mathcal{O} \\ K & \mathcal{O}
     \end{bmatrix}P^{-1}U\right)\left(U^{-1}P\begin{bmatrix}
        L_{\mathtt{r}}^{-1} & \mathcal{O} \\ X_{2} L_{\mathtt{r}}^{-1} & \mathcal{O}
    \end{bmatrix}Q\right)\left(Q\begin{bmatrix}
         L_{\mathtt{r}} & \mathcal{O} \\ K & \mathcal{O}
     \end{bmatrix}P^{-1}U\right)\\
     &=& Q\begin{bmatrix}
         L_{\mathtt{r}} & \mathcal{O} \\ K & \mathcal{O}
     \end{bmatrix}\begin{bmatrix}
        L_{\mathtt{r}}^{-1} & \mathcal{O} \\ X_{2} L_{\mathtt{r}}^{-1} & \mathcal{O}
    \end{bmatrix}\begin{bmatrix}
         L_{\mathtt{r}} & \mathcal{O} \\ K & \mathcal{O}
     \end{bmatrix}P^{-1}U=Q\begin{bmatrix}
         L_{\mathtt{r}} & \mathcal{O} \\ K & \mathcal{O}
     \end{bmatrix}P^{-1}U=S.
     \end{eqnarray*}
     Similarly, it can be shown that $S^{\{1,2\}}SS^{\{1,2\}}=S^{\{1,2\}}.$
     Since $X_2$ is arbitrary, the result follows:\[
    S^{\{1,2\}}=U^{-1}P\begin{bmatrix}
        L_{\mathtt{r}}^{-1} & \mathcal{O} \\ X_{2} L_{\mathtt{r}}^{-1} & \mathcal{O}
    \end{bmatrix}Q.\]
    \end{proof}
    \end{theorem}
The following algorithm provides a CRRMCF decomposition of $S=\mathcal{L}U$, thereby providing a direct procedure to compute the $\{1, 2\}$-inverse of a singular block-structure matrix $S$ using Theorem \ref{Theorem 3.1}. 
\begin{algorithm}[H] 
\caption{Computation of $\{1, 2\}-$inverse using $\mathcal{L}U$} \label{algo1}
\begin{algorithmic}[1]
\State \textbf{Input:} $A \in \mathbb{R}^{m \times n}.$
\State By definition \ref{embedding}, construct $S=\begin{bmatrix} S_{1} & S_{2} \\ S_{2} & S_{1} \end{bmatrix}.$
\Comment{Block-structure form}
\If{$\mathcal{R}(S_2) \subseteq \mathcal{R}(S_1)$}
\State Initialize $k \gets 0$, $U_{11} \gets I$, $\mathcal{S}_1 \gets S_1$ 
\Comment{Initialization of CRRMCF}
\For {$j=1$ to $\min\{m,n\}$}
\Comment{Loop over columns}
\If {column j of $\mathcal{S}_1$ is nonzero}
\State $k \gets k+1$
\State $c_k \gets j$
\Comment{Current pivot column index}
\State $r_k\gets$ pivot row index in $j^{th}$ column
\For {$i=c_k+1$ to $n$}
\Comment{Eliminate entries right to pivot}
\State Perform a column operation to get $\mathcal{S}_1(r_k,i)=0$ using $\mathcal{S}_1(r_k,c_k)$.
\State Store the multiplier $\dfrac{\mathcal{S}_1(r_k,i)}{\mathcal{S}_1(r_k,c_k)}$, used to make $\mathcal{S}_1(r_k,i)=0$ in $U_{11}(c_k,i).$
\EndFor
\EndIf
\EndFor
%\State Set $t_1 \leftarrow k$
\State Assign $\mathcal{L}_{11} \gets \mathcal{S}_1.$
\Comment{Derived CRRMCF of $S_1$}
\State Compute the CRRMCF decomposition of $S_1-S_2{U^{-1}_{11}}{\mathcal{L}^{\dag}_{11}}S_2=\mathcal{L}_{22}U_{22}.$
\State Calculate $\mathcal{L}_{21}=S_2U^{-1}_{11}$ and $U_{12}=\mathcal{L}^{\dag}_{11}S_2.$
\EndIf
\State Substituting the obtained expressions, we get the CRRMCF decomposition of $S$ as $S=
\begin{bmatrix}
\mathcal{L}_{11}& O\\ \mathcal{L}_{21} & \mathcal{L}_{22}\\
\end{bmatrix} 
\begin{bmatrix}
    U_{11} & U_{12} \\ O & U_{22}\\
\end{bmatrix}.$
\State Find permutation matrix $P$ and $Q$ such that $\mathcal{L}P=L$
   and $QL=\begin{bmatrix}
    L_{\mathtt{r}} & \mathcal{O} \\ K & \mathcal{O} \end{bmatrix}$, where $L_{\mathtt{r}} \in \mathbb{R}_{\mathtt{r}}^{r\times r}$ and $K \in \mathbb{R}^{(m-r)\times r}.$
    \State Compute $S^{\{1, 2\}}=\begin{bmatrix}
    U_{11}^{-1} & -U_{11}^{-1}U_{12}U_{22}^{-1} \\ O & U_{22}^{-1}\\
\end{bmatrix}P\begin{bmatrix}
      {L_{\mathtt{r}}}^{-1} & \mathcal{O} \\ \mathcal{O} & \mathcal{O}\\
    \end{bmatrix}Q.$
\end{algorithmic}
\end{algorithm}
\begin{remark}
    In practical problems, the test for determining pivot positions $(r_k,c_k)$ in the CRRMCF decomposition of any matrix $G$ (i.e., $S_1$ in line 6 and $S_1-S_2U_{11}^{-1}\mathcal{L}_{11}^{\dag} S_2$ in line 17 of Algorithm \ref{algo1})) uses a relative tolerance $\epsilon =\tau \|G\|_{\infty}$ \cite[p.276]{Matrix computation}, where $\tau \approx 2.22 \times 10^{-16}$ is the machine precision.
\end{remark}

The $\{1,2\}-$inverse can be employed to characterize the consistency of the crisp system associated with the FLS. The following result provides the necessary and sufficient condition for the existence of a solution.
\begin{theorem} \normalfont \label{Theorem 3.4} Consider the FLS given in (\ref{fls}) and let $SZ=B$ be the associated crisp system. Then the system $SZ=B$ is said to be {\it consistent} if and only if for some $S^{\{1,2\}}$, \begin{eqnarray}
 SS^{\{1,2\}}B=B \label{consis}  
\end{eqnarray} and the general solution is \begin{eqnarray}
Z=S^{\{1,2\}}B+(I-S^{\{1,2\}}S)h,\text{ for arbitrary }h\in \mathbb{R}^{2n} \label{consis_2}.\end{eqnarray} Moreover, if the obtained solution vector satisfies Definition \ref{weak and strong}, then the corresponding FLS admits a strong fuzzy solution. Otherwise, yields a weak fuzzy solution.
\begin{proof}
     Assume that $SZ=B$ is consistent. Then there exists some $y \in \mathbb{R}^{2n}$ such that $B=Sy$.
     Hence \[SS^{\{1,2\}}B=SS^{\{1,2\}}Sy=Sy=B.\]
     Conversely, if eqn. (\ref{consis}) holds, then by considering $Y=S^{\{1,2\}}B$, we have
     \[SY=SS^{\{1,2\}}B=B.\]
     Thus, $Y$ is a solution of $SZ=B,$ which gives the system to be consistent.
     From the given consistency condition, it is clear that $S^{\{1,2\}}B$ is a particular solution. Moreover, the homogeneous part of the solution directly follows from $SZ=O$ and $SS^{\{1,2\}}SZ=O$.
     \end{proof}
\end{theorem}
\iffalse
\begin{theorem} \normalfont
\textcolor{red}{ Consider a consistent system described in eqn. (\ref{fls}) and let $SZ=B$ be its associated  crisp system. For any $S^{\{1,2\}}$-inverse of $S$, $Z=S^{\{1,2\}}B+(I-S^{\{1,2\}}S)h$, for arbitrary $h\in \mathbb{R}^{2n}$ is a solution of the system and it admits either a strong or weak fuzzy solution.
%Moreover, if $h \in $ Null space of $S$, then the system admits the unique weak or strong fuzzy solution.\\
Proof follows immediately from Theorem \ref{Theorem 3.4} and Definition \ref{weak and strong}.}
\end{theorem}
\fi
In the above part, $\{1,2\}-$inverse was employed to solve the FLS when the associated crisp system is consistent. However, in many situation, the system may be inconsistent. In such situation it is necessary to determine a least squares solution. Thus, we construct $\{1,2,3\}-$inverse of the block structured matrix $S$.

\begin{theorem} \normalfont \label{(1,2,3)}
Suppose $S=LP^{-1}U$, obtained from Lemma \ref{L_U} and \ref{Lemma LPU} corresponding to a FLS (\ref{fls}). If $Q$ is a permutation matrix such that
$QL=\begin{bmatrix}
    L_{\mathtt{r}} & \mathcal{O} \\ K & \mathcal{O}
\end{bmatrix}$, then
    \[ 
    S^{\{1,2,3\}}=U^{-1}PE^{\{1,2,3\}}Q,
    \] where $L_{\mathtt{r}} \in \mathbb{R}_{\mathtt{r}}^{r\times r}, K \in \mathbb{R}^{(m-r)\times r}$, $E^{\{1,2,3\}}=\begin{bmatrix}
    L_{\mathtt{r}}^{-1}F^{-1} &L_{\mathtt{r}}^{-1}F^{-1}G^{T}\\ \mathcal{O} &\mathcal{O}\end{bmatrix}, F=I_{\mathtt{r}}+G^{T}G$ and G=$KL_{\mathtt{r}}^{-1}.$
\end{theorem}

\iffalse
\begin{theorem} \normalfont
   Suppose $A \in {\mathbb{R}}^{m\times n}$ and $S=LP^{-1}U$, obtained from Lemma \ref{L_U} and \ref{Lemma LPU}. If $Q$ is a permutation matrix such that
$QL=\begin{bmatrix}
    L_{\mathtt{r}} & \mathcal{O} \\ K & \mathcal{O}
\end{bmatrix}$, then
    \[ 
    S^{\{1,2,3\}}=U^{-1}P\begin{bmatrix}
        (L_{\mathtt{r}}^{T}L_{\mathtt{r}}+K^TK)^{-1} & (L_{\mathtt{r}}^{T}L_{\mathtt{r}}+K^TK)^{-1} \\ \mathcal{O} & X
    \end{bmatrix}\begin{bmatrix}
        L_{\mathtt{r}}^T & K^T \\\mathcal{O} & \mathcal{O}
    \end{bmatrix}Q,
    \]
    where $X$ is arbitrary matrix.
    \begin{proof}
     Since $L_{\mathtt{r}}$ is non-singular, the matrix $L_{\mathtt{r}}^{T}L_{\mathtt{r}}+K^TK$ is positive definite and hence $(L_{\mathtt{r}}^{T}L_{\mathtt{r}}+K^TK)^{-1}$ exists. Similar to Theorem \ref{(1,2,3)}, $SS^{\{1,2,3\}}S=S,S^{\{1,2,3\} }SS^{\{1,2,3\}}=S^{}$ and $(SS^{\{1,2,3\}})^{T}=SS^{\{1,2,3\}}$ can be verified.   
    \end{proof}
\end{theorem}
\fi 
%Since the normal form exist for all matrices, which implies the existence of $\{1,2,3\}-$inverse.
In order to obtain a least squares solution of an FLS, the $\{1,2,3\}$-inverse is computed using $\mathcal{L}U$ decomposition as illustrated in the following algorithm.
\begin{algorithm}[H] 
\caption{Fuzzy least squares solution of a singular FLS} \label{algo2}
\begin{algorithmic}[1]
    \State \textbf{Input:} An $n\times n$  FLS, $A\tilde{\mathbf{z}}=\tilde{\mathbf{b}}$.
    \State Implement Algorithm \ref{algo1} to obtain $S$.
    \State Compute $S^{\{1,2,3\}}$ using Theorem \ref{(1,2,3)} and $Z=S^{\{1,2,3\}}B$.
\end{algorithmic}
\end{algorithm}

The following theorem establishes the existence of strong or weak fuzzy least squares solutions for the inconsistent FLS.
\begin{theorem} \normalfont
    Let us consider the FLS given in (\ref{fls}) and let $SZ=B$ be the corresponding crisp system obtained. If $S^{\{1,2,3\}}$ is a $\{1,2,3\}-$inverse of $S,$ then the FLS (\ref{fls}) admits either strong or weak fuzzy least squares solution.
    \begin{proof}
       Suppose $S^{\{1,2,3\}}$ is a $\{1,2,3\}-$inverse of $S$ and let $Z=S^{\{1,2,3\}}B.$ Now consider the system $SZ=B$. Since $Z=S^{\{1,2,3\}}B$ which implies $SZ=SS^{\{1,2,3\}}B$. Multiplying $S^T$ on both sides, we get $S^TSZ=S^TSS^{\{1,2,3\}}B.$ Hence, $S^TSZ=S^T(SS^{\{1,2,3\}})^TB=S^T(S^{\{1,2,3\}})^TS^TB= (S S^{\{1,2,3\}} S)^T B= S^TB,$which is the normal equation associated with the crisp system. Therefore, $Z=S^{\{1,2,3\}}B$ gives a least square solution to the inconsistent system. By definition \ref{weak and strong}, FLS (\ref{fls}) corresponding to this admits either strong or weak fuzzy least squares solution.
    \end{proof}
\end{theorem}

Computing the $MP$-inverse is of particular interest for the matrix of the form (\ref{S}), as demonstrated in the next result, which is motivated by the $MP$-inverse of a partitioned matrix of Ching-hsiang Hung and Markham \cite{Hun}.
\begin{theorem} \label{Theorem 3.9} \normalfont
   Consider a matrix $A \in \mathbb{R}^{m\times n}$, and a block-structure matrix $S$ of the form (\ref{S}) with $\mathcal{R}(S_2) \subseteq \mathcal{R}(S_1)$. Let \[
    S=\begin{bmatrix}
    \mathcal{L}_{11} & O \\ \mathcal{L}_{21} & \mathcal{L}_{22}\\
\end{bmatrix}
\begin{bmatrix}
    U_{11} & U_{12} \\ O & U_{22}\\
\end{bmatrix}
    \] be the CRRMCF decomposition of $S$, then the $MP$-inverse is given by
    \begin{eqnarray}
    S^{\dag}=\begin{bmatrix}
    B^{\dag}[(\mathcal{L}_{11}U_{11})^{T}-C(H+F^{\dag}D^{T})] & B^{\dag}[(\mathcal{L}_{21}U_{11})^{T}-C(K+F^{\dag}E^{T})] \\ H+F^{\dag}D^{T} & K+F^{\dag}E^{T}
    \end{bmatrix},\label{main_eqn}
 \end{eqnarray}
    where 
    \begin{eqnarray}
             B&=&(\mathcal{L}_{11}U_{11})^{T}\mathcal{L}_{11}U_{11}+(\mathcal{L}_{21}U_{11})^{T}\mathcal{L}_{21}U_{11} \nonumber\\
        C&=&(\mathcal{L}_{11}U_{11})^{T}\mathcal{L}_{11}U_{12}+(\mathcal{L}_{21}U_{11})^{T}(\mathcal{L}_{21}U_{12}+\mathcal{L}_{22}U_{22}) \nonumber\\
        D&=&\mathcal{L}_{11}U_{12}-\mathcal{L}_{11}U_{11}B^{\dag}C \nonumber\\
        E&=&(\mathcal{L}_{21}U_{12}+\mathcal{L}_{22}U_{22})-\mathcal{L}_{21}U_{11}B^{\dag}C \nonumber\\
        F&=&D^{T}D+E^{T}E \nonumber\\
        G&=&B^{\dag}C(I-F^{\dag}F) \nonumber\\
        H&=&(I-F^{\dag}F)(I+G^{T}G)^{-1}(B^{\dag}C)^{T}B^{\dag}((\mathcal{L}_{11}U_{11})^{T}-CF^{\dag}D^{T}) \nonumber\\
        K&=&(I-F^{\dag}F)(I+G^{T}G)^{-1}(B^{\dag}C)^{T}B^{\dag}((\mathcal{L}_{21}U_{11})^{T}-CF^{\dag}E^{T}).
    \end{eqnarray}
    \begin{proof} \normalfont
    The $MP$-inverse of $S$ is going to be computed by using Theorem \ref{Theorem 3.8}. The procedures are derived as follows.\\
    Step 1: Constructing $W$ and $V$.\\
    Suppose \[
W=\begin{bmatrix}
\mathcal{L}_{11}U_{11} & O \\ \mathcal{L}_{21}U_{11} & O
\end{bmatrix} ~\mbox{and}~
V=\begin{bmatrix}
    O & \mathcal{L}_{11}U_{12} \\ O & \mathcal{L}_{21}U_{12}+\mathcal{L}_{22}U_{22}
\end{bmatrix}.
\]
Clearly it satisfies $WV^{T}=O$.
By Theorem \ref{Theorem 3.8}, $S^{\dag}$ can be expressed as
 \begin{eqnarray}S^{\dag}=W^{\dag}+(I-W^{\dag}V)[M^{\dag}+(I-M^{\dag}M)NV^{T}(W^{\dag})^{T}W^{\dag}(I-VM^{\dag})] \label{final},\end{eqnarray}
    where 
    \begin{eqnarray}
        M&=&(I-WW^{\dag})V \label{M}\\
        N&=&[I+(I-M^{\dag}M)V^{T}(W^{\dag})^{T}W^{\dag}V(I-M^{\dag}M)]^{-1} \label{3.7}.
    \end{eqnarray}
Step 2: Computing of $W^{\dag}$ and $M$\\
Since the second block-column of $W$ is zero, using Theorem \ref{Lemma 3.7}, we obtain
\begin{eqnarray}
W^{\dag}=\begin{bmatrix}
B^{\dag}(\mathcal{L}_{11}U_{11})^{T} & B^{\dag}(\mathcal{L}_{21}U_{11})^{T} \nonumber \\ O & O
\end{bmatrix},
\end{eqnarray}  
where $B=(\mathcal{L}_{11}U_{11})^{T}\mathcal{L}_{11}U_{11}+(\mathcal{L}_{21}U_{11})^{T}\mathcal{L}_{21}U_{11}.$ Using eqn. (\ref{M}), we compute
\begin{eqnarray}
M=V-WW^{\dag}V
=\begin{bmatrix}
   O & \mathcal{L}_{11}U_{12}-\mathcal{L}_{11}U_{11}B^{\dag}C \\ O & (\mathcal{L}_{21}U_{12}+\mathcal{L}_{22}U_{22})-\mathcal{L}_{21}U_{11}B^{\dag}C 
\end{bmatrix}=\begin{bmatrix}
        O & D \\ O & E
    \end{bmatrix}, \label{3.8}
\end{eqnarray}
where the intermediate block matrices are defined as \begin{eqnarray*}
    C&=&(\mathcal{L}_{11}U_{11})^{T}\mathcal{L}_{11}U_{12}+(\mathcal{L}_{21}U_{11})^{T}(\mathcal{L}_{21}U_{12}+\mathcal{L}_{22}U_{22}) \nonumber\\
        D&=&\mathcal{L}_{11}U_{12}-\mathcal{L}_{11}U_{11}B^{\dag}C \nonumber\\
        E&=&(\mathcal{L}_{21}U_{12}+\mathcal{L}_{22}U_{22})-\mathcal{L}_{21}U_{11}B^{\dag}C. \nonumber
\end{eqnarray*}
Step 3: Computation of $M^{\dag}$ and $N$.\\
Applying Theorem \ref{Lemma 3.6} on $M^T$ and by using the property $((M^{T})^T)^{\dag}=((M^{T})^{\dag})^T=M^{\dag},$ we obtain $MP$-inverse of $M$ as $M^{\dag}=\begin{bmatrix}
  O & O \\ F^{\dag}D^T &  F^{\dag}E^T
\end{bmatrix},$
where the matrix $F$ is defined by $F=D^{T}D+E^{T}E.$ Next, consider the expression
\begin{eqnarray*}
    W^{\dag}V(I-M^{\dag}M)&=&
     \begin{bmatrix}
        O & B^{\dag}C \\ O & O
    \end{bmatrix}
     \begin{bmatrix}
        I & O \\ O & I-F^{\dag}F
    \end{bmatrix}
    =
     \begin{bmatrix}
        O & G \\ O & O
    \end{bmatrix},\\
\end{eqnarray*}
where $G=B^{\dag}C(I-F^{\dag}F).$
Thus, we obtain the final structure of $N$ as
\begin{eqnarray}
N&=&\begin{bmatrix}
\begin{bmatrix}
        I & O \\ O & I
    \end{bmatrix}+ 
    \begin{bmatrix}
        I & O \\ O & I-F^{\dag}F
    \end{bmatrix}
    \begin{bmatrix}
        O &  O\\ (B^{\dag}C)^{T} & O
    \end{bmatrix}
     \begin{bmatrix}
        O & G \\ O & O
    \end{bmatrix}
\end{bmatrix}^{-1} \nonumber  
=
\begin{bmatrix}
        I & O \\ O & I+G^{T}G
    \end{bmatrix}^{-1}.  \label{3.9}
    \end{eqnarray}
    Moreover, $G^{T}G$ is a positive semi-definite matrix, which implies $(I+G^TG)^{-1}$ exists.\\
Step 4: Final expression of $S^{\dag}$.\\
By combining steps 2 and 3, we obtain
\begin{eqnarray*}
W^{\dag}(I-VM^{\dag})&=&\begin{bmatrix}
B^{\dag}(\mathcal{L}_{11}U_{11})^{T} & B^{\dag}(\mathcal{L}_{21}U_{11})^{T} \\ O & O
\end{bmatrix}-
\begin{bmatrix}
        O & B^{\dag}C \\ O & O
    \end{bmatrix}
    \begin{bmatrix}
  O & O \\ F^{\dag}D^T &  F^{\dag}E^T
\end{bmatrix}\\
&=&\begin{bmatrix}
B^{\dag}[(\mathcal{L}_{11}U_{11})^{T}-CF^{\dag}D^{T}] &B^{\dag}[(\mathcal{L}_{21}U_{11})^{T}-CF^{\dag}E^{T}] \\ O & O
\end{bmatrix}.
\end{eqnarray*}
Hence, on substituting this expression, we get
\begin{eqnarray} \label{3.10}
M^{\dag}+(I-M^{\dag}M)NV^{T}(W^{\dag})^{T}W^{\dag}(I-VM^{\dag})=\begin{bmatrix}
O & O\\ H+F^{\dag}D^{T} & K+F^{\dag}E^{T}
\end{bmatrix},
\end{eqnarray}
where $H=(I-F^{\dag}F)(I+G^{T}G)^{-1}(B^{\dag}C)^{T}B^{\dag}((\mathcal{L}_{11}U_{11})^{T}-CF^{\dag}D^{T})$ and\\ $K=(I-F^{\dag}F)(I+G^{T}G)^{-1}(B^{\dag}C)^{T}B^{\dag}((\mathcal{L}_{21}U_{11})^{T}-CF^{\dag}E^{T})$.
Finally substituting all the components into eqn. (\ref{final}), we obtain
\begin{eqnarray*}
    S^{\dag}
    &=&
    \begin{bmatrix}
B^{\dag}(\mathcal{L}_{11}U_{11})^{T} & B^{\dag}(\mathcal{L}_{21}U_{11})^{T} \\ O & O
\end{bmatrix}+
\begin{bmatrix}
I & -B^{\dag}C \\ O & I
\end{bmatrix}
\begin{bmatrix}
O & O\\ H+F^{\dag}D^{T} & K+F^{\dag}E^{T}
\end{bmatrix}\\
&=&
\begin{bmatrix}
B^{\dag}[(\mathcal{L}_{11}U_{11})^{T}-C(H+F^{\dag}D^{T})] & B^{\dag}[(\mathcal{L}_{21}U_{11})^{T}-F(K+F^{\dag}E^{T})] \\ H+F^{\dag}D^{T} & K+F^{\dag}E^{T}
\end{bmatrix},
\end{eqnarray*}
which completes the proof.
\end{proof}
\end{theorem}
The result thus obtained in Theorem \ref{Theorem 3.9} can be further simplified under additional range conditions on the matrix $S$. The following corollary provides a reduced structure of the MP-inverse.
\begin{corollary}\normalfont
   Consider a matrix $A \in \mathbb{R}^{m\times n}$, and $S=LP^{-1}U$ be the block structure obtained from Lemma \ref{L_U} and \ref{Lemma LPU}. If $Q$ is a permutation matrix such that
$QL=\begin{bmatrix}
    L_{\mathtt{r}} & \mathcal{O} \\ K & \mathcal{O}
\end{bmatrix}$ and $\mathcal{R}(UU^T\mathcal{L}^T) \subseteq \mathcal{R}(\mathcal{L}^T)$ then the $MP-$inverse using  Theorem \ref{Theorem 3.9} is given as
    \[ 
    S^{\dag}=U^{-1}P\begin{bmatrix}
    B^{-1}{L_{\mathtt{r}}}^{T} & B^{-1}K^T \\ \mathcal{O} & \mathcal{O}
    \end{bmatrix}Q,
    \]where $L_{\mathtt{r}} \in \mathbb{R}_{\mathtt{r}}^{r\times r}, K \in \mathbb{R}^{(m-r)\times r}$, $B= {L_{\mathtt{r}}}^TL_{\mathtt{r}}+K^TK$.
\end{corollary}
The following theorem provides a direct method for $MP-$inverse of the block structured matrix with full-column rank.
\begin{theorem} \normalfont
    Suppose the block structured matrix $S \in \mathbb{R}^{2m \times 2n}$ corresponding to the FLS (\ref{fls}) is of full-column rank with $\mathcal{R}(S_2) \subseteq \mathcal{R}(S_1)$. Let $S=LP^{-1}U$  be the decomposed structure obtained from Lemma \ref{L_U} and \ref{Lemma LPU}, then  the $MP-$inverse of $S$ is given $S^{\dag}=U^{-1}P(L^TL)^{-1}L^T.$ 
    \begin{proof}
        Since $U$ is non-singular and $S$ is of full-column rank, which implied $L$ is of full column rank. Thus, we obtain \begin{eqnarray*}
         S^{\dag}&=&(LP^{-1}U)^{\dag}=((LP^{-1}U)^TLP^{-1}U)^{-1}(LP^{-1}U)^T=(U^TPL^TLP^{-1}U)^{-1}(LP^{-1}U)^T\\&=&U^{-1}P(L^TL)^{-1}P^{-1}(U^T)^{-1}U^TPL^T=U^{-1}P(L^TL)^{-1}L^T.
        \end{eqnarray*}
    \end{proof}
\end{theorem}
\iffalse
\begin{theorem} \normalfont
\textcolor{red}{Suppose $A \in {\mathbb{R}}^{m\times n}$ and $S=LP^{-1}U$, obtained from Lemma \ref{L_U} and \ref{Lemma LPU}. If $Q$ is a permutation matrix such that
$QL=\begin{bmatrix}
    L_{\mathtt{r}} & \mathcal{O} \\ K & \mathcal{O}
\end{bmatrix}$ and $\mathcal{R}(UU^T\mathcal{L}^T) \subseteq \mathcal{R}(\mathcal{L}^T)$ then
    \[ 
    S^{\dag}=U^{-1}PE^{\dag}Q,
    \] where $L_{\mathtt{r}} \in \mathbb{R}_{\mathtt{r}}^{r\times r}, K \in \mathbb{R}^{(m-r)\times r}$, $E^{\dag}=\begin{bmatrix}
    L_{\mathtt{r}}^{-1}F &L_{\mathtt{r}}^{-1}F(KL_{\mathtt{r}}^{-1})^{T}\\ \mathcal{O} &\mathcal{O}\end{bmatrix}$ and $F=[I_{\mathtt{r}}+(KL_{\mathtt{r}}^{-1})^{T}(KL_{\mathtt{r}}^{-1})]^{-1}$.
    \begin{proof}
        Since $\mathcal{R}(UU^T\mathcal{L}^T) \subseteq \mathcal{R}(\mathcal{L}^T)$ and $U$ is non-singular which implies $(\mathcal{L}U)^{\dag}=U^{-1}\mathcal{L}^{\dag}$. The proof follows along the same lines as Theorem \ref{(1,2,3)}. \end{proof}}
    \end{theorem}
\fi
Although, we present the computation of generalized inverses using $\mathcal{L}U$ decomposition, this relies on specific range conditions. To deal general rectangular matrices without any restrictions, we constructed an alternative direct decomposition. The following theorem provides the computation of $MP-$inverse using full rank decomposition without any range conditions. Initially, we begin by constructing a full rank decomposition of $S.$ Since computing the full rank decomposition of a $2m \times 2n$ matrix takes a high operational count, we propose a modified construction in which two full rank decomposition of $n \times n$ matrices are required instead. \\To construct this decomposition of the block structured matrix $S$ defined in eqn. (\ref{S}), we consider the full rank decomposition of $S_{1}$ and $S_1-S_2(S_1)^{\dag}S_2=F_{22}G_{22}$, given by $S_1=F_{11}G_{11}$ and $S_1-S_2(S_1)^{\dag}S_2=F_{22}G_{22},$ where $F_{11} \in \mathbb{R}^{m \times {\mathtt{r}}_1}$, $F_{22} \in \mathbb{R}^{m \times {\mathtt{r}}_2}$, $G_{11} \in \mathbb{R}^{{\mathtt{r}}_1 \times n}$ and $G_{22} \in \mathbb{R}^{{\mathtt{r}}_2 \times n}$, ${\mathtt{r}}_1=rank(S_1)$ and ${\mathtt{r}}_2=rank(S_1-S_2(S_1)^{\dag}S_2)$. Using the above construction, we obtain the following result. 
\begin{theorem} \normalfont
Let us consider a $2m \times 2n$ block structured matrix $S=\begin{bmatrix}
    S_1 & S_2 \\ S_2 & S_1
\end{bmatrix}$ corresponding to a $m \times n$ FLS defined in (\ref{fls}). Let $S_1=F_{11}G_{11}$ and $S_1-S_2(S_1)^{\dag}S_2=F_{22}G_{22}$ be a full rank decompositions. Then, the full rank decomposition of $S$ is  \begin{eqnarray}
S=\mathcal{F} \mathcal{G}, \label{FG}    
\end{eqnarray} where $\mathcal{F}=\begin{bmatrix}
    F_{11} & \mathcal{O} \\ F_{21} & F_{22}
\end{bmatrix} \in \mathbb{R}^{2m \times ({\mathtt{r}}_1+{\mathtt{r}}_2)}\text{ and }\mathcal{G}=\begin{bmatrix}
    G_{11} & G_{12} \\ \mathcal{O} & G_{22}
\end{bmatrix}\in \mathbb{R}^{({\mathtt{r}}_1+{\mathtt{r}}_2) \times 2n}.$
%The off-diagonal matrices are obtained by solving $S_2=F_{21}G_{11}$ and $S_2=F_{11}G_{12}$ and hence it is given as \[F_{21}=S_2G_{11}^T(G_{11}G_{11}^T)^{-1} \text{ and }G_{12}=(F_{11}^TF_{11})^{-1}F_{11}^TS_2.\]
Consequently, the $MP-$inverse of $S$ is given as \begin{eqnarray}
S^{\dag}=\mathcal{G}^{\dag}\mathcal{F}^{\dag},    
\end{eqnarray} where \begin{eqnarray*}
    \mathcal{F}^{\dag}&=&\begin{bmatrix}
    A^{-1}_{1}F_{11}^T-A^{-1}_{1}F_{21}^TF_{22}D^{-1}_{1}B^T_{1} & A^{-1}_{1}F_{21}^T-A^{-1}_{1}F_{21}^TF_{22}D^{-1}_{1}C^T_{1} \\ D^{-1}_{1}B^T_{1} & D^{-1}_{1}C^T_{1}
\end{bmatrix}, \\
\mathcal{G}^{\dag}&=& \begin{bmatrix}
G_{11}^TA^{-1}_{2}-B^T_{2}D^{-1}_{2}G_{22}G_{12}^TA^{-1}_{2} & B^T_{2}D^{-1}_{2} \\ G_{12}^TA^{-1}_{2}-C^T_{2}D^{-1}_{2}G_{22}G_{12}^TA^{-1}_{2} & C^T_{2}D^{-1}_{2}
\end{bmatrix}
\end{eqnarray*}
and the blocks $A_{1},B_{1},C_{1},D_{1},A_{2},B_{2},C_{2},D_{2}$ are defined by:
\begin{eqnarray*}
 &&A_{1}=F^T_{11}F_{11}+F^T_{21}F_{21},~~~~~~~~B_{1}=-F_{11}A^{-1}_{1}F^T_{21}F_{22},\\
 &&C_{1}=F_{22}-F_{21}A^{-1}_{1}F^T_{21}F_{22},~~~ D_{1}=B^T_{1}B_{1}+C^T_{1}C_{1},\\
 &&A_{2}=G_{11}G^T_{11}+G_{12}G^T_{12},~~~~~~ B_{2}=-G^T_{11}A^{-1}_{2}G_{12}G_{22}^T,\\
 &&C_{2}=G_{22}^T-G^T_{12}A^{-1}_{2}G_{12}G^T_{22},~ D_{2}=B^T_{2}B_{2}+C^T_{2}C_{2}.
\end{eqnarray*}
\begin{proof}
    In order to construct a full rank decomposition of $S$, we construct the block matrices $\mathcal{F}$ and $\mathcal{G}$ as follows: \[\begin{bmatrix}
    S_1 & S_2 \\ S_2 & S_1
\end{bmatrix}=\begin{bmatrix}
    F_{11} & \mathcal{O} \\ F_{21} & F_{22}
\end{bmatrix}\begin{bmatrix}
    G_{11} & G_{12} \\ \mathcal{O} & G_{22}
\end{bmatrix}.\]
 Comparing the corresponding blocks, we get
\begin{eqnarray}
    S_1 &=& F_{11}G_{11},~~S_2 = F_{21}G_{11} \implies F_{21}=S_2 G^{\dag}_{11}, \nonumber\\
    S_2 &=& F_{11}G_{12} \implies G_{12}= F^{\dag}_{11}S_2,~~S_1 = F_{21}G_{12}+F_{22}G_{22}. \label{FG1} 
    \end{eqnarray}  
Since $S_1=F_{11}G_{11}$ is a full rank decomposition,  the off-diagonal block matrices $F_{21}$ and $G_{12}$ are obtained as \[F_{21}=S_2 G_{11}^T(G_{11}G_{11}^T)^{-1} \text{ and }G_{12}=(F_{11}^T F_{11})^{-1}F_{11}^T S_2.\]  Now, consider the block matrix $\mathcal{F}=\begin{bmatrix} 
    F_{11} & \mathcal{O} \\ F_{21} & F_{22}
\end{bmatrix}$. As $F_{11}$ and $F_{22}$ are of full column rank and $\mathcal{F}$ is a lower triangular block matrix with upper right block as zero, the columns of $\begin{bmatrix}
    F_{11} \\ F_{21}
\end{bmatrix}$ and $\begin{bmatrix}
    \mathcal{O} \\ F_{22}
\end{bmatrix}$ are linearly independent. Therefore, the obtained $\mathcal{F}$ is a full column rank matrix. Similarly, the block matrix $\mathcal{G}=\begin{bmatrix}
    G_{11} & G_{12} \\ \mathcal{O} & G_{22}
\end{bmatrix}$ is a full row rank matrix. Therefore, by definition of full rank decomposition, the obtained decomposition $S=\mathcal{F}\mathcal{G}$ is a full rank decomposition of the block structured matrix $S.$ Using, the approach of Theorem \ref{Theorem 3.9}, we obtain $\mathcal{F}^{\dag}$. As, $D_{1}$ is the Schur complement of $A_{1}$ in the non-singular matrix $\mathcal{F}^T\mathcal{F}$, so $D_{1}^{-1}$ exists. Similarly, $\mathcal{G}^{\dag}$ follows from $\mathcal{G}^{\dag}=((\mathcal{G}^T)^{\dag})^T$.
\end{proof}
\end{theorem}  
\
Using the $MP-$inverse obtained from the above theorems, Theorem \ref{minimum-norm} provides the existence of strong or weak minimum norm least squares solution.
     \begin{theorem} \normalfont \label{minimum-norm}
        Consider the FLS (\ref{fls}) and let $SZ=B$ be the corresponding crisp system. Then the fuzzy system (\ref{fls}) admits a strong or weak fuzzy minimum norm least squares solution  $Z=S^{\dag}B+(I-S^{\dag}S)h,$ h is arbitrary.
    \end{theorem}
Theorem \ref{Theorem strong 1} and \ref{Theorem strong 2} offer the necessary and sufficient condition of the system to have a strong least squares solution. In the following theorems, $(S_1+S_2)^{GI}$ denote the generalized inverse of the matrix $S_1+S_2,$ which includes $\{1,2,3\}-$inverse and $MP-$inverse.
\begin{theorem} \normalfont \label{Theorem strong 1}
    Consider an FLS (\ref{fls}) and assume that $S$ is of the form (\ref{S}). Then the system (\ref{fls}) has a strong least squares solution if and only if the following conditions hold for every $\alpha \in [0,1]$:
    \begin{enumerate}
       \item[(i)]  $(S_1+S_2)^{GI}(\overline{b}(\alpha)-\underline{b}(\alpha))\geq \mathbf{O}$
        \item[(ii)] $(S_1+S_2)^{GI}\frac{d}{d \alpha}(\underline{b}(\alpha)) \geq \mathbf{O}$
        \item[(iii)] $(S_1+S_2)^{GI}\frac{d}{d \alpha}(\overline{b}(\alpha)) \leq \mathbf{O}.$
    \end{enumerate} 
    \begin{proof}
   Consider $SZ=B$.
    On expanding, we have
    \begin{eqnarray}
    S_1 \underline{z}(\alpha) - S_2 \overline{z}(\alpha) &=& \underline{b}(\alpha), \nonumber \\
    S_2 \underline{z}(\alpha) - S_1 \overline{z}(\alpha) &=& -\overline{b}(\alpha). \label{strong}
    \end{eqnarray}
Combining these equations, it follows that $(S_1+S_2)(\overline{z}(\alpha)-\underline{z}(\alpha))=(\overline{b}(\alpha)-\underline{b}(\alpha))$, which implies
    \begin{eqnarray}(\overline{z}(\alpha)-\underline{z}(\alpha))=(S_1+S_2)^{GI}(\overline{b}(\alpha)-\underline{b}(\alpha)). \label{converse}
   \end{eqnarray}
   Also (ii) and (iii) hold for all $\alpha \in [0,1],$ which implies $\underline{z}(\alpha)$ and $\overline{z}(\alpha)$ are non-decreasing and non-increasing functions respectively over $[0,1]$.
    Moreover $(S_1+S_2)^{GI}(\overline{b}(\alpha)-\underline{b}(\alpha))\geq \mathbf{O},$ which gives $(\overline{z}(\alpha)-\underline{z}(\alpha)) \geq \mathbf{O}$. Conversely, from eqn. (\ref{converse}) and (\ref{strong}), the result is obvious.
\end{proof}
\end{theorem}

\begin{theorem} \normalfont \label{Theorem strong 2}
    Let us consider an FLS (\ref{fls}) and assume that $S$ is of the form (\ref{S}). Then the system (\ref{fls}) has a strong least squares solution if and only if the following conditions hold for every $\alpha_1,\alpha_2 \in [0,1]$:
    \begin{enumerate}
       \item[(i)]  $(S_1+S_2)^{GI}(\overline{b}(\alpha_1)-\underline{b}(\alpha_1))\geq \mathbf{O}$
        \item[(ii)] for any $\alpha_1<\alpha_2$, $(S_1+S_2)^{GI}(\overline{b}(\alpha_1)-\overline{b}(\alpha_2)) \geq \mathbf{O}$
        \item[(iii)] for any $\alpha_1<\alpha_2$, $(S_1+S_2)^{GI}(\underline{b}(\alpha_1)-\underline{b}(\alpha_2)) \leq \mathbf{O}$.
    \end{enumerate} 
    The proof follows similarly to that of Theorem \ref{Theorem strong 1}.
    \end{theorem}
    \iffalse
The following theorem provides a sufficient condition for the FLS to have strong fuzzy solution.
\begin{theorem} \normalfont \label{Theorem strong 3}
    \textcolor{red}{Let us consider an FLS (\ref{fls}) and assume that $S$ is of the form (\ref{S}). Then the system (\ref{fls}) has a strong minimum norm least squares solution if the rows or columns of $|A|$ has rank 1 blocks.}
    \begin{proof}
    From eqn (\ref{converse}), we have \[(\overline{z}-\underline{z})=(S_1+S_2)^{\dag}(\overline{b}-\underline{b}).\] Since the rows or columns of $|A|=S_1+S_2$ have rank 1 blocks, which implies $(S_1+S_2)^{\dag} \geq 0 (\cite{Berman}).$ Also $\tilde{b}$ is a fuzzy vector, which implies \[\overline{z}-\underline{z} \geq 0.\] Moreover, $\overline{b}$ and $\underline{b}$ are monotonically decreasing and monotonic increasing function, respectively, which implies $\overline{z}$ and $\underline{z}$ are monotonically decreasing and monotonic increasing function, respectively. Therefore, from definition (\ref{weak and strong}), $\tilde{z}$ is a strong fuzzy solution.
    \end{proof}
    \end{theorem}
\fi
    %In order to find the full rank decomposition of $S,$ we must find a full column rank matrix $\mathcal{F}\in \mathbb{R}^{2m \times r}$ and a full row rank matrix $\mathcal{G} \in \mathbb{R}^{r \times 2n}$ such that \[S=\mathcal{F} \mathcal{G},\] where $r=rank(S).$
The obtained above results provides the theoretical methodology for computation of $\{1,2\}-$inverse, $\{1,2,3\}-$inverse and $MP-$inverse. However, in practice, efficient computation of the $MP-$inverse is essential to obtain the unique minimum norm least squares solution. Thus, we establish algorithmic approaches based on matrix decompositions for practical implementation by converting the block matrix $S$ into a simpler form using orthogonal matrices. First, we present a $\mathcal{L}U$ decomposition algorithm using CRRMCF for computing $MP-$inverse under specific range conditions.
\begin{algorithm}[H]
\caption{$\mathcal{L}U$ decomposition method to compute $MP$-inverse} \label{algo6}
\begin{algorithmic}[1]
    \State Input: An $m\times n$  FLS, $A\tilde{\mathbf{z}}=\tilde{\mathbf{b}}$.
    \State Apply definition \ref{embedding} to transforms into a crisp system (\ref{SZ=B}).
    \State Construct $P_1^{T}S P_2$, where $P_1=\frac{1}{\sqrt{2}}\begin{bmatrix}
    I_m & -I_m \\ I_m & I_m
\end{bmatrix}$ and $P_2=\frac{1}{\sqrt{2}}\begin{bmatrix}
    I_n & -I_n \\ I_n & I_n
\end{bmatrix}$.
\State Compute the $\mathcal{L}U$ decomposition of $S_1+S_2=\mathcal{L}_1U_1$ and $S_1-S_2=\mathcal{L}_2U_2$.
\State Find 
permutation matrices $Q_1$ and $Q_2$ such that $S_1+S_2=Q_1^{T}L_1U_1$ and $S_1-S_2=Q_2^{T}L_2U_2$.
\If{ $\mathcal{R}(U_iU_i^T\mathcal{L}_i^T)\subseteq \mathcal{R}(\mathcal{L}_i^T), i=1,2$,}
\State The $LU$ decomposition of $P_1^{T}SP_2$ is $\begin{bmatrix}
L_1 & O \\O  & L_2  
\end{bmatrix}\begin{bmatrix}
Q_1^{T} & O \\O  & Q_2^{T}  
\end{bmatrix} \begin{bmatrix}
    U_1 & O \\ O & U_2
\end{bmatrix}.$
\State Compute $S^{\dag}=P_1\begin{bmatrix}
U_1^{-1} & O \\O  & U_2^{-1}  
\end{bmatrix}\begin{bmatrix}
    Q_{1} & O \\ O & Q_{2}
\end{bmatrix}\begin{bmatrix}
L_{1}^{\dag} & O \\ O & L_{2}^{\dag} 
\end{bmatrix}P_2^T.$
\EndIf
\end{algorithmic}
\end{algorithm}
Although, the $MP-$ inverse of $S$ can be derived directly using $LU$ decomposition Algorithm \ref{algo6} based on CRRMCF and depends on the range condition. Thus, orthogonal based decomposition such as $QR$ decomposition and singular valued decomposition ($SVD$) are often employed in numerical computations due to robustness in solving rectangular systems. Therefore, in addition to the $\mathcal{L}U$ approach, we also proposed $QR$ and $SVD$ based algorithms for computing the $MP-$inverse without any specific conditions. Now, we propose an algorithm to compute $QR$ decomposition of the block-structure matrix $S$.
\begin{algorithm}[H]
\caption{$QR$ decomposition method to compute $MP$-inverse} \label{algo3}
\begin{algorithmic}[1]
 \State Using Algorithm \ref{algo6}, determine $P_1^{T}SP_2$.
\State Compute the $QR$ decomposition of $S_1+S_2=Q_1R_1$ and $S_1-S_2=Q_2R_2$.
\State The $QR$ decomposition of $P_1^{T}SP_2$ is $\begin{bmatrix}
Q_1 & O \\O  & Q_2  
\end{bmatrix} \begin{bmatrix}
    R_1 & O \\ O & R_2
\end{bmatrix}.$
\State Compute $S^{\dag}=P_1\begin{bmatrix}
R_{1}^{\dag} & O \\ O & R_{2}^{\dag} 
\end{bmatrix}
\begin{bmatrix}
    Q_{1}^T & O \\ O & Q_{2}^T
\end{bmatrix}P_2^T.$
\end{algorithmic}
\end{algorithm}
In the following algorithm we propose $SVD$ to compute the $MP$-inverse of a block-structure matrix.
\begin{algorithm}[H] 
\caption{Least squares method for $SVD$ decomposition} \label{algo4}
\begin{algorithmic}[1]
    \State Using Algorithm \ref{algo6}, determine $P_1^{T}SP_2$.
    \State Compute $SVD$ for $S_1+S_2=U_1 \sum_{1} V_{1}^T$ and $S_1-S_2=U_2 \sum_{2} V_2^{T}$.
\State The $SVD$ for $P_1^{T}SP_2$ is $\begin{bmatrix}
U_1 & O \\O  & U_2  
\end{bmatrix} \begin{bmatrix}
    \sum_1 & O \\ O & \sum_2
\end{bmatrix}\begin{bmatrix}
V_{1}^T & O \\O  & V_{2}^T  
\end{bmatrix}.$
\State Determine $S^{\dag}=P_1\begin{bmatrix}
V_1 & O \\O  & V_2  
\end{bmatrix}\begin{bmatrix}
\sum_{1}^{\dag} & O \\ O & \sum_{2}^{\dag} 
\end{bmatrix}
\begin{bmatrix}
U_{1}^T & O \\O  & U_{2}^T  
\end{bmatrix}P_2^T.$
\end{algorithmic}
\end{algorithm}

In the following remark, we provide the operation count of the proposed direct method and its comparison with existing schemes for solving a $2n \times 2n$ crisp system.
\begin{remark} \label{OR}
    The proposed Algorithms \ref{algo6}, \ref{algo3}, \ref{algo4} are compared with the existing work of Hosseinpour \cite{svd}, where $SVD$ was employed to compute the $MP$-inverse. As reported in \cite{Matrix computation}, the operation count for their method is $48m^2n+64mn^2+72n^3$+$O(n^2).$ In contrast, the proposed methods for block-structure matrix $S$ yield a significantly lower operation count for computing the $MP$-inverse of the matrix $S$, given by
\begin{table}[H]
    \centering
\begin{tabular}{c c}
\hline
   Proposed Methods  & Operation Counts \\
   \hline
    $\mathcal{L}U$ & $6mr^2+6r^3+8nr^2+\frac{2}{3}n^3+O(n^2)$\\
    \hline
    $QR$ & $8mr^2+8nr^2+\frac{16}{3}r^3+O(n^2)$\\
     \hline
    $SVD$ & $12m^2n+16mn^2+18n^3+O(n^2)$ \\
    \hline
\end{tabular}
    \caption{Operation Count for $2m \times 2n$ matrix $S$}
\end{table}
Here, $O(n^2)$ represents the collection of all lower-order terms whose value is bounded by constant times $n^2$ and $\mathtt{r}=\min\{m,n\}$. For matrices satisfying the given range condition, $\mathcal{L}U$ can be employed because it exhibits lower computational cost. For matrices that violates the range condition, the $QR$ can be used. SVD can be used in case of handling ill-conditioned matrices. Although it has high computational cost, it provides good numerical stability for computing the solution. 
\end{remark}
%Furthermore, the operation counts for both the proposed method and existing method are evaluated, and a performance for a $1000000 \times 1000000$ is displayed in Example \ref{OR10000}.
In the next section, we have provided a few numerical examples to implement the proposed methods. The appropriate generalized inverse is computed using the direct decompositions of the block matrix as illustrated in the proposed algorithms. 
\section{Numerical Experiments} \label{Numerical Example}
All computations, including operation counts and execution times, were processed using MATLAB R2024a on a system with an Intel(R) Core(TM) i5 processor (2.10 GHz) with 16 GB RAM. Consider a singular crisp linear system of the form $SZ=B$. A vector $Z$ is said to be a least squares solution if the residual $R=SZ-B$ has the minimal Euclidean norm, {\it i.e.,} $\|R\|_2$ is minimum. The minimum norm least squares solution is given by $Z_0=S^{\dag}B$, and it satisfies the following properties: $\|SZ_0-B\| \leq \|SZ-B\|,$ for all $Z$, and $\|Z_0\|<\|Z\|,$
for any $Z \neq Z_0$. The {\it relative error} and {\it relative residual} are defined as 
$Relative~Error~(\mathcal{R}_{E})= \frac{\|Z-\tilde{Z}\|_2}{{\|Z\|}_{2}}$ and $
Relative~Residual ~(\mathcal{R}_{R})= \frac{\|B-S\tilde{Z}\|_2}{\|B\|_2}$,
where $\tilde{Z}$ denotes an approximation of the exact solution $Z$. In the following computation the exact solution is taken as the $MP$-inverse obtained in MATLAB using the command \texttt{pinv}.\\
 The established mathematical findings are now applied to practical problems where  imprecision is inherent. To illustrate this, first we consider an analog complex RLC circuit example with fuzzy voltages. Analog circuits often involve imprecise parameters such as resistance, capacitance, or voltage due to environmental factors or measurement errors. Representing these parameters as fuzzy numbers provides a more realistic analysis of a circuit's behavior.
\begin{example} \label{circuit} \label{example 3}
The following figure represents a complex analog electrical circuit consisting of voltage, resistors, capacitors, and inductors. Let $\tilde{I}_1$, $\tilde{I}_2$, $\tilde{I}_3$, $\tilde{I}_4$, and $\tilde{I}_5$ denote the current flowing in the clockwise direction. The voltages are represented using trapezoidal fuzzy numbers $(9,11,13,15)$, $(3,5,7,9)$, and $(16,18,20,22)$ to capture the uncertainties in their values caused by external factors.
\begin{figure}[H]
\centering
\begin{circuitikz} 
\draw[thick,red] (2,-5)--(5,-5);
\draw[thick,red] (2,0)--(5,0);
\draw[thick,red] (2,-5)--(2,-3);
\draw[thick,blue] (2.5,-2.5) arc[start angle=0, end angle=360, radius=0.5 cm];
\draw[thick,red] (2,-2)--(2,0);
\node at (2,-2.8){$-$};
\node at (2,-2.305){$+$};
\node[rotate=90] at (1.3,-2.5){$(9, 11, 13,15)$};
\draw[<-] (4.2,-2.5) arc[start angle=0, end angle=230, radius=0.6 cm];
\node at (3.6,-2.59){$\tilde{I}_{1}$};

\draw[thick,red] (5,-5)--(8,-5);
\draw[thick,red] (5,0)--(8,0);
\draw[thick,red] (5,-5)--(5,-4.5);
\draw[thick,blue] (5,-4.5) to [C](5,-4.2);
\draw[thick,red] (5,-4.2)--(5,-3.4);
\draw[thick,blue] (5,-3.4) to [L](5,-2.4);
\draw[thick,red] (5,-2.4)--(5,-1.7);
\draw[thick,blue] (5,-1.7) to [american resistor, R](5,-0.6);
\draw[thick,red] (5,-0.6)--(5,0);
\draw[thick,red] (8,-5)--(8,-3);
\draw[thick,blue] (8.5,-2.5) arc[start angle=0, end angle=360, radius=0.5 cm];
\draw[thick,red] (8,-2)--(8,0);
\node at (8,-2.8){$+$};
\node at (8,-2.305){$-$};
\node at (5.9,-4.3){$-5i\Omega$};
\node at (5.6,-2.9){$2i\Omega$};
\node at (5.6,-1.1){$7\Omega$};
\node[rotate=90] at (8.8,-2.5){$i(3,5,7,9)$};
\draw[<-] (7.2,-2.5) arc[start angle=0, end angle=230, radius=0.6 cm];
\node at (6.6,-2.59){$\tilde{I}_{2}$};
\draw[<-] (10.4,-2.5) arc[start angle=0, end angle=230, radius=0.6 cm];
\node at (9.8,-2.59){$\tilde{I}_{3}$};

\draw[thick,red] (8,-5)--(11,-5);
\draw[thick,red] (10,0)--(11,0);
\draw[thick,red] (8,0)--(9.5,0);
\draw[thick,blue] (9.5,0) to [C](10,0);
\draw[thick,red] (11.5,-5)--(11.5,-4.6);
\draw[thick,blue] (11.5,-4.6) to [american resistor, R](11.5,-3.5);
\draw[thick,red] (11.5,-3.5)--(11.5,-2.8);
\draw[thick,blue] (11.5,-2.8) to [L](11.5,-1.8);
\draw[thick,red] (11.5,-1.8)--(11.5,-1);
\draw[thick,blue] (11.5,-1) to [C](11.5,-0.6);
\draw[thick,red] (11.5,-0.6)--(11.5,0);
\node at (12.1,-4.){$3\Omega$};
\node at (12.1,-2.25){$6i\Omega$};
\node at (12.4,-0.8){$-i\Omega$};
\node at (9.6,0.6){$-i\Omega$};

\draw[thick,red] (11,-5)--(15,-5);
\draw[thick,red] (11,0)--(12.7,0);
\draw[thick,blue] (12.7,0) to [L](13.6,0);
\draw[thick,red] (13.6,0)--(15,0);
\draw[thick,red] (15,-5)--(15,-4);
\draw[thick,blue] (15,-4) to [C](15,-3.5);
\draw[thick,red] (15,-3.5)--(15,-2);
\draw[thick,blue] (15,-2) to [R](15,-0.9);
\draw[thick,red] (15,-0.9)--(15,0);
\node at (13.1,0.5){$4i\Omega$};
\node at (15.6,-1.3){$4\Omega$};
\node at (15.6,-3.3){$-3i\Omega$};
\draw[<-] (14,-2.5) arc[start angle=0, end angle=270, radius=0.6 cm];
\node at (13.4,-2.59){$\tilde{I}_{4}$};
\draw[thick,red] (8,0)--(8,0.75);
\draw[thick,red] (8,1.75)--(8,2.5);
\draw[thick,red] (15,0)--(15,2.5);
\draw[thick,blue] (8.5,1.25) arc[start angle=0, end angle=360, radius=0.5 cm];
\draw[<-] (12,1.25) arc[start angle=0, end angle=270, radius=0.5 cm];
\node at (11.5,1.25){$\tilde{I}_{5}$};
\draw[thick,red] (8,2.5)--(9.5,2.5);
\draw[thick,blue] (9.5,2.5)to [C](10,2.5);
\draw[thick,red] (10,2.5)--(11.1,2.5);
\draw[thick,blue] (11.1,2.5) to [R](12.2,2.5);
\draw[thick,red] (12.2,2.5)--(13,2.5);
\draw[thick,blue] (13,2.5)to [L](13.9,2.5);
\draw[thick,red] (13.9,2.5)--(15,2.5);
\node at (9.1,2.9){$-i\Omega$};
\node at (11.6,3){$9\Omega$};
\node at (13.4,3){$3i\Omega$};
\node[rotate=90] at (7.2,1.3){$(16,18,20,22)$};
\node at (8,1.0){$-$};
\node at (8,1.5){$+$};
\end{circuitikz}
\caption{A complex RLC circuit with fuzzy current and voltage}\label{circuitpicture}
\end{figure}
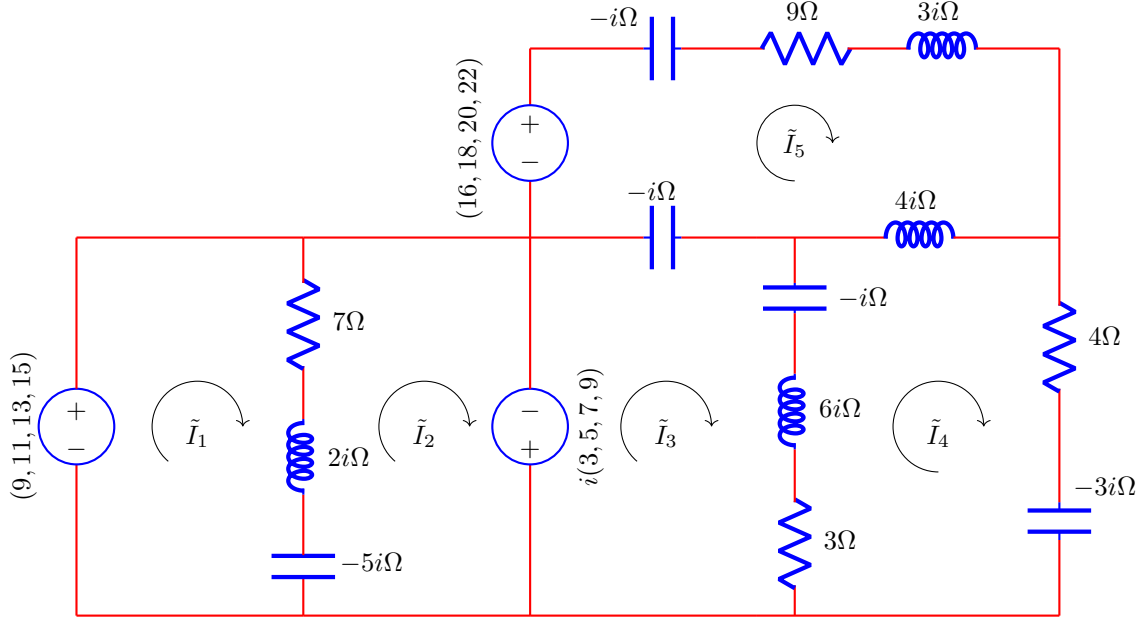
Applying Kirchhoff's voltage law in loops $\tilde{I}_{1}, \tilde{I}_{2}, \tilde{I}_{3}, \tilde{I}_{4}$, and $\tilde{I}_{5}$, we get
\begin{eqnarray}
   (7-3i)\tilde{I}_{1}+(-7+3i)\tilde{I}_{2}&=&(9, 11, 13,15)\\
   (-7+3i)\tilde{I}_{1}+(7-3i)\tilde{I}_{2}&=&i(3,5,7,9)\\
   (3+4i)\tilde{I}_{3}+(-3-5i)\tilde{I}_{4}+i\tilde{I}_{5}&=&i(-9,-7,-5,-3)\\
   (-3-5i)\tilde{I}_{3}+(7+6i)\tilde{I}_{4}-4i\tilde{I}_{5}&=&0\\
   i\tilde{I}_{3}-4i\tilde{I}_{4}+(9+5i)\tilde{I}_{5}&=&(16,18,20,22),
\end{eqnarray}
where $\tilde{I}_{j}=\tilde{p}_{j}+i\tilde{q}_{j},~j=1,2,3,4,5.$ This system of equations can be equivalently written as 
\begin{eqnarray*}
&\begin{bmatrix}
      7-3i & -7+3i & 0& 0& 0\\
      -7+3i & 7-3i & 0 & 0& 0\\
      0 & 0 & 3+4i & -3-5i& i \\
      0&0&-3-5i&7+6i&-4i\\
      0&0&i&-4i&9+5i
    \end{bmatrix}
    \begin{bmatrix}
        \tilde{I}_1 \\\tilde{I}_2\\\tilde{I}_3\\\tilde{I}_4\\\tilde{I}_5
    \end{bmatrix}=
    \begin{bmatrix}
     (9,11,13,15)\\ 
     i(3,5,7,9)\\
     i(-9,-7,-5,-3)\\
     0\\
     (16,18,20,22)
    \end{bmatrix}.\\
    &i.e., ~A\tilde{\mathbf{z}}=\tilde{\mathbf{b}}.
\end{eqnarray*}
This $5\times5$ complex FLS is transformed into a $10\times10$ real FLS according to the approach in \cite{Complex}, and the resulting structure is expressed as
\[
\begin{bmatrix}
    7 & -7& 0& 0& 0& 3& -3& 0& 0& 0 \\ -7& 7& 0& 0& 0& -3& 3& 0& 0& 0 \\ 0& 0& 3 &-3 &0 &0 &0 &-4 &5 &-1 \\ 0& 0& -3 &7 &0 &0 &0 &5 &-6 &4 \\ 0& 0& 0& 0& 9& 0& 0& -1& 4& -5 \\ -3 &3 &0 &0 &0 &7 &-7 &0 &0 &0 \\ 3& -3& 0& 0& 0& -7& 7& 0& 0& 0 \\ 0& 0& 4& -5& 1& 0& 0& 3& -3& 0 \\ 0& 0& -5& 6& -4& 0& 0& -3& 7& 0 \\ 0& 0& 1& -4& 5& 0& 0& 0& 0& 9
\end{bmatrix}
\begin{bmatrix}
 \tilde{p}_{1} \\ \tilde{p}_{2} \\ \tilde{p}_{3} \\ \tilde{p}_{4} \\ \tilde{p}_{5} \\ \tilde{q}_{1} \\ \tilde{q}_{2} \\\tilde{q}_{3} \\\tilde{q}_{4} \\\tilde{q}_{5}
\end{bmatrix}=
\begin{bmatrix}
 (9+2\alpha,15-2\alpha)\\  0\\ 0\\ 0 \\ (16+2\alpha,22-2\alpha)\\0\\(3+2\alpha,9-2\alpha) \\(-9+2\alpha,-3-2\alpha) \\0 \\0
\end{bmatrix}, 0 \leq \alpha \leq 1.
\]
Since the obtained system is a real FLS and the associated crisp linear system is inconsistent. By using Algorithm \ref{algo2}, the solution in terms of $\{1,2,3\}$-inverse is obtained.
\iffalse, and the final form is expressed as
\[\begin{bmatrix}
( \frac{1}{10}\alpha+\frac{423}{580},-\frac{1}{10}\alpha+\frac{597}{580})+i( \frac{1}{10}\alpha-\frac{117}{580},-\frac{1}{10}\alpha+\frac{57}{580})\\
(0,0)+i(0,0)\\
(0.4164\alpha-1.0997,-0.4164\alpha+0.1495)+i(1.5836\alpha-3.6319,-1.5836\alpha+1.1189)\\
(0.2629\alpha+0.2945,-0.2629\alpha+1.0832)+i(1.0705\alpha-2.0415,-1.0705\alpha+1.1699)\\
(0.8118\alpha+0.5892,-0.8118\alpha+3.0246)+i(0.5216\alpha-1.4272,-0.5216\alpha+0.1374)
\end{bmatrix}.
\]
\fi

Figure \ref{fig:plot3} illustrates the weak fuzzy solution of the complex-valued system $A\tilde{\mathbf{z}}=\tilde{\mathbf{b}},$ with the real part $\tilde{p}_j=(\underline{p}_j(\alpha),\overline{p}_j(\alpha)),j=1,2,3,4,5,$ displayed in the left panel and the imaginary part $\tilde{q}_j=(\underline{q}_j(\alpha),\overline{q}_j(\alpha)),j=1,2,3,4,5,$ displayed in the
right panel for different values of $\alpha \in [0,1]$. Each curve corresponds to the parametric form of the weak fuzzy solution. The first two current components, $\tilde{I}_1$ and $\tilde{I}_2$, exhibit the crisp intervals independent of $\alpha$, whereas the currents $\tilde{I}_3,\tilde{I}_4$, and $\tilde{I}_5$ represent TrFN. This demonstrates that the proposed method effectively models the propagation of fuzziness and complex behavior in the electrical network.
 \begin{figure}[H]
    \centering
\includegraphics[height=2.5in,width=3.1in]{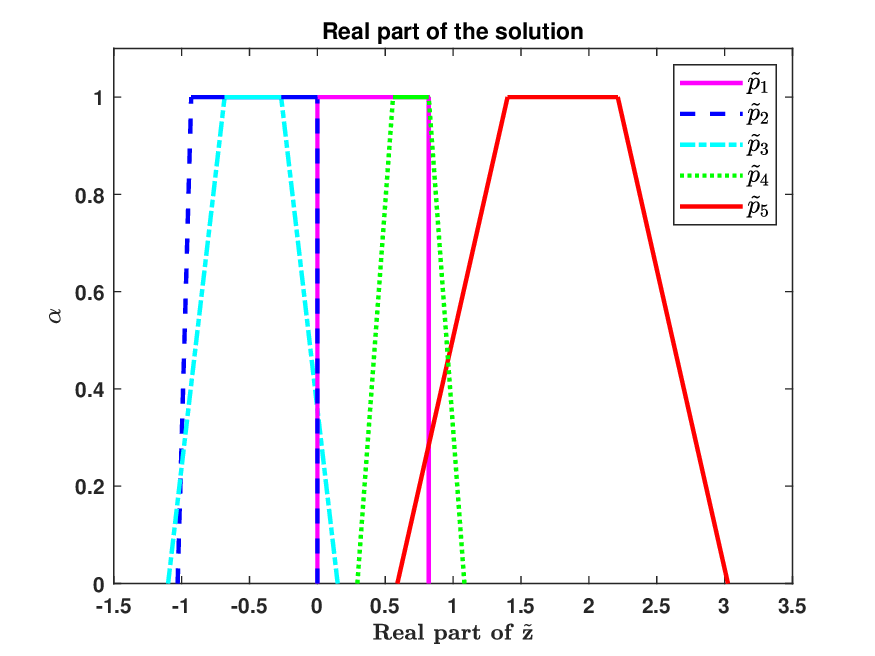}
     \includegraphics[height=2.5in,width=3.1in]{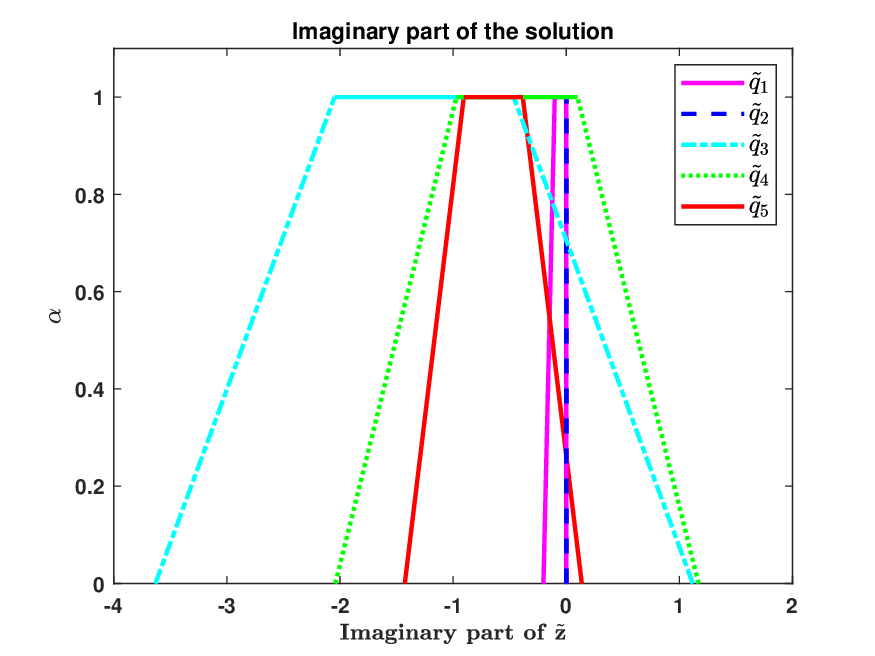}
    \caption{Solution Graph of the Complex FLS}
    \label{fig:plot3}
\end{figure}
\end{example}
Next, we consider a Markov decision process with fuzzy rewards. In practical applications, the reward function associated with a particular state in a Markov Decision Process (MDP) is not always precise due to environmental fluctuations or ambiguity. Addressing these problems provides a framework for analyzing the random processes, such as the movement of gas molecules in a chamber and the navigation of robots along narrow paths. In this example, the rewards are typically represented by using Triangular fuzzy numbers, which effectively models imprecision.
The increasing fuzzy rewards show the goal significance, which helps in resolving the practical issues. 
\begin{example} [\cite{Berman}] \label{example 1} Consider a motion in which a particle is suspended in a gas that moves along a straight path in unit steps. Assuming that a particle moves to the right with a probability p=0.7 and to the left with a probability q=1-0.7=0.3 when it collides. When it attains the extreme position, it reverts one step with probability 1. 
\begin{figure}[H]
    \centering
   \begin{tikzpicture}
   \node at (0,0) (S1) {$S_1$};
    \node at (2,0) (S2) {$S_2$};
     \node at (4,0) (S3) {$S_3$};
      \node at (5,0) (dots) {$\cdots$};
      \node at (6,0) (Sn-1) {$S_{n-1}$};
      \node at (8,0) (Sn) {$S_n$};
      
     \draw[->,blue] (S1) to[bend left=40] node[above] {$1$} (S2);
      \draw[<-,red] (S1) to[bend right=40] node[below] {$q$} (S2);
      \draw[->,blue] (S2) to[bend left=40] node[above] {$p$} (S3);
      \draw[<-,red] (S2) to[bend right=40] node[below] {$q$} (S3);

      \draw[->,blue] (Sn-1) to[bend left=40] node[above] {$p$} (Sn);
      \draw[<-,red] (Sn-1) to[bend right=40] node[below] {$1$} (Sn);
   \end{tikzpicture}
    \caption{Markov Chain with Transition Probabilities}
\end{figure}
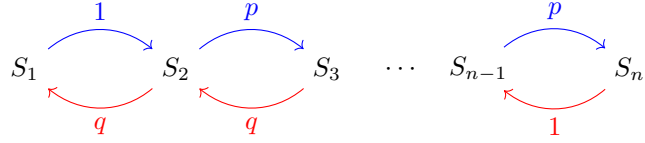
The corresponding transition matrix $\mathcal{T}$ of order $n=4$ is given by
$\mathcal{T}=\begin{bmatrix}
    0 & 1 & 0 & 0\\q & 0 & p & 0 \\ 0 & q & 0 & p \\ 0 & 0 & 1 & 0   
   \end{bmatrix}.$
   Consider Bellman's matrix equation 
   \begin{eqnarray} \label{Bellman_equation}
   (I-\gamma \mathcal{T})\tilde{\mathbf{v}}=\tilde{\mathbf{r}},      
   \end{eqnarray} where $\mathcal{T}$ is a transition matrix, $\tilde{\mathbf{r}}$ is a fuzzy reward vector, $\tilde{\mathbf{v}}$ is a value function, and $\gamma \in [0,1]$ is a discount factor. Let the reward assigned be $\tilde{\mathbf{r}}=\begin{bmatrix}
          (2,1,1)\\
          (3,1,1) \\
          (4,1,1) \\
        (5,1,1) \end{bmatrix}.$
        Here, the fuzzy rewards are assigned based on the position of a particle in a particular state $S_{i},i=1,2,3,4$. Low (high) fuzzy rewards are assigned in the starting (final) state, while the intermediate state receives increasing fuzzy rewards, with the aim of reflecting an increasing significance of reaching a goal. When $\gamma <1$, the influence of future reward increases gradually with an increasing value of $\gamma$, and the value function $\tilde{\mathbf{v}}$ remains finite. At $\gamma=1$, all future rewards are considered equal, which results in the matrix $(I-\gamma \mathcal{T})$ to be singular. By substituting $\gamma$ as 1 in the FLS (\ref{Bellman_equation}), we obtain 
   \begin{eqnarray*}
   \begin{bmatrix}
   1&-1&0&0\\
   -0.3&1&-0.7&0\\
   0&-0.3&1&-0.7\\
   0&0&-1&1 
\end{bmatrix}\begin{bmatrix}
       \tilde{v_1} \\ \tilde{v_2} \\ \tilde{v_3} \\\tilde{v_4}
   \end{bmatrix}&=&\begin{bmatrix}
          (2,1,1)\\
          (3,1,1) \\
          (4,1,1) \\
        (5,1,1) \end{bmatrix}.\label{markov}
   \end{eqnarray*}
   The extended $8 \times 8$ system $SV=R$ is expressed as
   \begin{eqnarray} \label{8by8}
\begin{bmatrix}
1&0&0&0&0&1&0&0\\0&1&0&0&0.3&0&0.7&0\\0&0&1&0&0&0.3&0&0.7\\0&0&0&1&0&0&1&0\\0&1&0&0&1&0&0&0\\0.3&0&0.7&0&0&1&0&0\\0&0.3&0&0.7&0&0&1&0\\0&0&1&0&0&0&0&1
\end{bmatrix}
\begin{bmatrix}
    \underline{v}_{1}(\alpha) \\ \underline{v}_{2}(\alpha) \\
    \underline{v}_{3}(\alpha) \\
    \underline{v}_{4}(\alpha) \\
    -\overline{v}_{1}(\alpha)\\
    -\overline{v}_{2}(\alpha)\\
    -\overline{v}_{3}(\alpha)\\
    -\overline{v}_{4}(\alpha)\\
\end{bmatrix}&=&
\begin{bmatrix}
1+\alpha\\2+\alpha\\3+\alpha\\4+\alpha\\-3+\alpha\\-4+\alpha\\-5+\alpha\\-6+\alpha
\end{bmatrix}, 0 \leq \alpha \leq 1.
\end{eqnarray}
By applying Algorithm \ref{algo2}, the explicit form of the solution takes the form 
\[
\begin{bmatrix}
    \underline{v}_{1}(\alpha) \\ \underline{v}_{2}(\alpha) \\
    \underline{v}_{3}(\alpha) \\
    \underline{v}_{4}(\alpha) \\
    -\overline{v}_{1}(\alpha)\\
    -\overline{v}_{2}(\alpha)\\
    -\overline{v}_{3}(\alpha)\\
    -\overline{v}_{4}(\alpha)\\
\end{bmatrix}=\begin{bmatrix}
  0.9566+0.5000\alpha\\
  -0.3555+0.5000\alpha\\
  -1.9280+0.5000\alpha\\
  -0.6731+0.5000\alpha\\
  -1.9566+0.5000\alpha\\
  -0.6445+0.5000\alpha\\
   0.9280+0.5000\alpha\\
  -0.3569+0.5000\alpha
\end{bmatrix}.
\]
By Theorem \ref{Theorem strong 1}, the obtained solution is a strong fuzzy solution. Figure \ref{fig:plot1} illustrates the strong fuzzy solution $\tilde{\mathbf{v}}$ obtained using $\mathcal{L}U$ decomposition of $S$. Clearly, each $\tilde{v}_i,~i=1,2,3,4,$ corresponds to an Triangular fuzzy number in a parametric form with respect to $\alpha$.
\begin{figure}[H]
    \centering    
    \includegraphics[height=3.1in,width=5in]{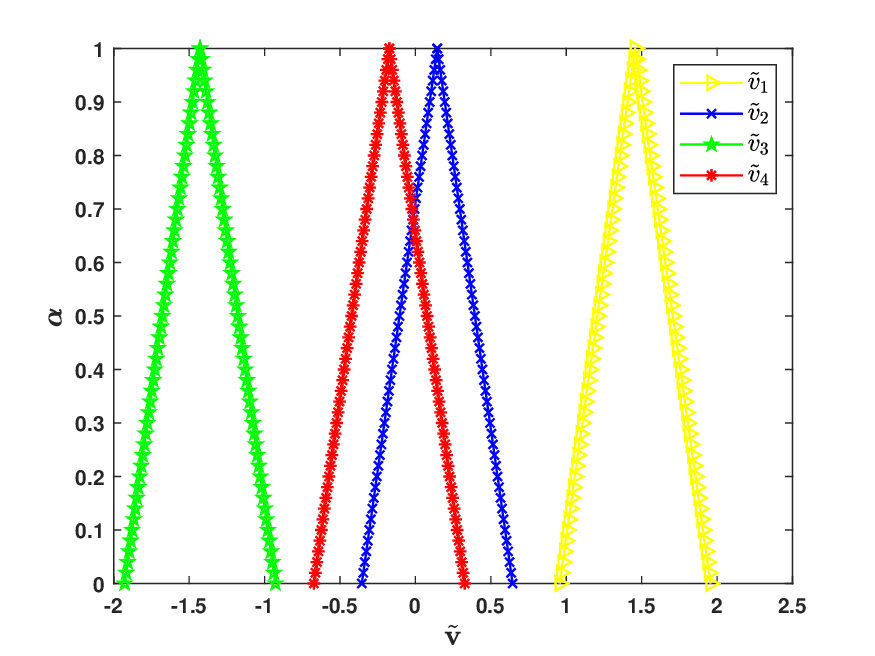}
    \caption{Solution of a singular FLS}
    \label{fig:plot1}
\end{figure}
Thus, the final form of the solution using $\mathcal{L}U$ decomposition is given by
\[
\begin{bmatrix}
       \tilde{v_1} \\ \tilde{v_2} \\ \tilde{v_3} \\\tilde{v_4}
   \end{bmatrix}=\begin{bmatrix}
(0.9566+0.5000\alpha,1.9566-0.5000\alpha)\\
(-0.3555+0.5000\alpha,0.6445-0.5000\alpha)\\
(-1.9280+0.5000\alpha,-0.9280-0.5000\alpha)\\
(-0.6731+0.5000\alpha,0.3269-0.5000\alpha)
\end{bmatrix}.
\]
Similarly, by applying Algorithm \ref{algo3} and \ref{algo4}, $QR$ and $SVD$ yield the same solution.
In Table \ref{tab:placeholder} we present the mean computational time required to find the minimum norm least squares solution of the proposed direct methods, including $\mathcal{L}U, QR,$ and $SVD$, Hosseinpour \cite{svd}, and MATLAB, for different values of $\alpha$ within the interval $[0,1]$. Moreover, the least squares solution of these direct methods is obtained via Algorithms \ref{algo2}, \ref{algo3}, and \ref{algo4}. In this example, the coefficient matrix is considered as an extension of the above transition matrix $\mathcal{T}$ and the reward $\tilde{\mathbf{r}}=(r_i)$, is defined as $r_i=(-1+\alpha,1-\alpha), \alpha \in [0,1]$ and $i=1,2,\cdots,n$. In this problem, all computations are carried out with fixed $\gamma$ value of 1.
\begin{table}[H]
    \centering
    \caption{Comparative Analysis of Mean Computational Time of Direct Methods.}
    \label{tab:placeholder}
    \begin{tabular}{c c c c c c c}
    \hline
   \multirow{2}{*}{\shortstack{Size of a\\ Matrix}} & \multirow{2}{*}{Methods} & \multicolumn{5}{c}{Mean Time (sec)} \\ 
\cline{3-7}
 & & $\alpha=0$ & $\alpha=0.25$ & $\alpha=0.5$ & $\alpha=0.75$ & $\alpha=1$ \\
\hline
    $100 \times 100 $ & \makecell{$\mathcal{L}U$ (proposed) \\ $QR$ (proposed) \\ $SVD$ (proposed)\\Hosseinpour \cite{svd}\\ MATLAB} & \makecell{0.3673\\0.3627\\0.5139\\0.3681\\0.3840} & \makecell{0.3504\\0.3466\\0.4837\\0.3811\\0.3762} & \makecell{0.3523\\0.3466\\0.4813\\0.3778\\0.3762} & \makecell{0.3508\\0.3451\\0.4826\\0.3883\\0.3723} & \makecell{0.2882\\0.2670\\0.3684\\0.2896\\0.2980}\\
    \hline
    $200 \times 200 $ & \makecell{$\mathcal{L}U$ (proposed) \\ $QR$ (proposed) \\ $SVD$ (proposed)\\Hosseinpour \cite{svd}\\ MATLAB} & \makecell{1.3644\\0.8132\\1.2930\\1.4016\\1.3723} & \makecell{1.3354\\0.8158\\1.2690\\1.3744\\1.3575} & \makecell{1.3420\\0.8036\\1.2741\\1.3945\\1.3567} & \makecell{1.3301\\0.8035\\1.2816\\1.3952\\1.3534} & \makecell{1.0237\\0.6408\\0.9798\\1.1072\\1.0810}\\
    \hline
    $500 \times 500 $ & \makecell{$\mathcal{L}U$ (proposed) \\ $QR$ (proposed) \\ $SVD$ (proposed)\\Hosseinpour \cite{svd}\\ MATLAB} & \makecell{8.5879\\7.0097\\8.2631\\8.8631\\9.2021} & \makecell{8.4016\\7.0901\\8.2893\\8.9579\\8.7559} & \makecell{8.5386\\7.1484\\8.3842\\9.0236\\8.8751} & \makecell{8.5386\\7.2374\\8.3885\\9.0399\\9.2001} & \makecell{6.4856\\5.6882\\6.5538\\6.7821\\6.6794}\\
    \hline
    $1000 \times 1000 $ & \makecell{$\mathcal{L}U$ (proposed) \\ $QR$ (proposed) \\ $SVD$ (proposed)\\Hosseinpour \cite{svd}\\ MATLAB} & \makecell{50.5868\\39.8236\\37.7119\\63.7851\\64.4268} & \makecell{50.6748\\39.9505\\37.8283\\60.3205\\63.4311} & \makecell{51.5023\\40.2237\\37.3139\\64.7233\\63.9479} & \makecell{55.6680\\40.1737\\37.3052\\60.5895\\85.1639} & \makecell{40.7696\\32.6497\\30.7157\\47.5036\\75.6219}\\
    \hline
    $2000 \times 2000 $ & \makecell{$\mathcal{L}U$ (proposed) \\ $QR$ (proposed) \\ $SVD$ (proposed)\\Hosseinpour \cite{svd}\\ MATLAB} & \makecell{271.3835\\141.7869\\167.2870\\451.3327\\367.3540} & \makecell{304.0993\\156.7859\\167.8046\\571.2192\\398.1755} & \makecell{309.2561\\179.7081\\192.5695\\699.9123\\451.4055} & \makecell{316.8527\\179.9787\\197.3244\\312.9890\\480.6870} & \makecell{241.9585\\138.8589\\151.1812\\247.2937\\375.7968}\\
    \hline
    \end{tabular}
\end{table}

Table \ref{tab:placeholder} reports the mean computational cost for different sizes of matrices under varying values of $\alpha$. In all the methods, the minimum computational time occurred when $\alpha=1$, while slightly higher values occur at $\alpha=0.25$ and $0.5$. The $QR$ decomposition consistently exhibits the lowest mean computational time; {\it i.e.,} a $100 \times 100$ matrix requires 0.3627 seconds at $\alpha=0$ and 0.2670 seconds at $\alpha=1$, and a $2000 \times 2000$ matrix requires 138.8589 seconds at $\alpha=1$ and 141.7869 seconds at $\alpha=0$. Following $QR$, $SVD$ shows a lower mean computational time. For instance, in $1000 \times 1000$ matrix, $SVD$ takes approximately 32.6497 seconds at $\alpha=1$ and 39.8236 seconds at $\alpha=0$, whereas all other approaches require approximately between 40.7696 and 75.6219 seconds across different values of $\alpha$. Although $\mathcal{L}U$ decomposition is computationally slow, it still produces the reduced cost of time when compared to Hosseinpour \cite{svd} and MATLAB. Overall, $QR$ demonstrates the most efficient performance, producing the lowest mean computational time for all $\alpha.$ This analysis, indicates that the proposed methods highlight the effectiveness in handling large-scale systems.\par 
To further validate the proposed techniques for large matrices, we examine the mean relative error ($\mathcal{R}_{ME}$) and mean relative residual ($\mathcal{R}_{MR}$) in determining the minimum norm least squares solution. The relative error signifies the deviation of the obtained solution from the exact solution, while the residual measures the consistency of the computed solution with the given system. The results for various sizes of matrices are summarized in Table \ref{table1}. The reward $\tilde{\mathbf{r}}=(r_i)$, in this computation is considered as $r_i=(-1+\alpha,1-\alpha)$, where $\alpha \in [0,1)$ and $i=1,2,\cdots,n$.
\begin{table}[H]
\centering
\caption{Analysis of $\mathcal{R}_{ME}$ and $\mathcal{R}_{MR}$ while Computing the Minimum Norm Least Squares Solution.}\label{table1}
\begin{tabular}{c c c c} 
\hline
   Size of a Matrix & Proposed Methods  & $\mathcal{R}_{ME}$ & $\mathcal{R}_{MR}$\\
   \hline
   $100 \times 100$ & \makecell{$\mathcal{L}U$ \\ $QR$ \\ $SVD$ } & \makecell{$0.2164 e^{-12}$ \\ $0.5815 e^{-13}$ \\$0.5581e^{-13}$} & \makecell{$0.9361 e^{-14}$ \\ $0.3526 e^{-14}$ \\$0.2307 e^{-14}$} \\ 
   \hline
   $200 \times 200$ &  \makecell{$\mathcal{L}U$ \\ $QR$ \\ $SVD$ } & \makecell{$0.5418 e^{-12}$ \\ $0.9170 e^{-13}$ \\$0.9233 e^{-13}$} & \makecell{$0.1703 e^{-13}$ \\ $0.4230 e^{-14}$ \\$0.3203 e^{-14}$} \\ 
   \hline
   $500 \times 500$ &  \makecell{$\mathcal{L}U$ \\ $QR$ \\ $SVD$ }& \makecell{$0.9567 e^{-12}$ \\ $0.1581 e^{-12}$ \\$0.1556 e^{-12}$} & \makecell{$0.5997 e^{-13}$ \\ $0.6669 e^{-14}$ \\$0.4498 e^{-14}$} \\ 
   \hline
   $1000 \times 1000$ & \makecell{$\mathcal{L}U$ \\ $QR$ \\ $SVD$ } & \makecell{$0.1693 e^{-11}$ \\ $ 0.2069 e^{-12}$ \\$ 0.2042e^{-12}$} & \makecell{$ 0.1617e^{-12}$ \\ $ 0.7751e^{-14}$ \\$ 0.6184e^{-14}$} \\ 
   \hline
   $2000 \times 2000$ & \makecell{$\mathcal{L}U$ \\ $QR$ \\ $SVD$ } & \makecell{$0.5191 e^{-11}$ \\ $ 0.2846 e^{-12}$ \\$ 0.3245e^{-12}$} & \makecell{$ 0.4646e^{-12}$ \\ $ 0.1136e^{-13}$ \\$ 0.8447e^{-14}$} \\ 
   \hline
\end{tabular}
\end{table}

In Table \ref{table1}, we analyze $\mathcal{R}_{ME}$ and $\mathcal{R}_{MR}$ for the proposed direct methods across systems of varying sizes, ranging from $100 \times 100$ to $2000 \times 2000$. For smaller $100 \times 100$ and $200 \times 200$ systems, $QR$ and $SVD$ outperforms $LU$, with $SVD$ yielding the lowest $\mathcal{R}_{MR}$. For instance, in $100 \times 100$ matrix, $SVD$ yields $\mathcal{R}_{ME}=0.5581e^{-13}$ and $\mathcal{R}_{MR}=0.2307e^{-14}$, while $QR$ yields $\mathcal{R}_{ME}=0.5815e^{-13}$ and $\mathcal{R}_{MR}=0.3526e^{-14}$. In contrast, $\mathcal{L}U$ achieves larger values with $\mathcal{R}_{ME}=0.2164e^{-12}$ and $\mathcal{R}_{MR}=0.9361e^{-14}$. As the size increases to $1000 \times 1000$ and $2000 \times 2000$, the superiority of $QR$ and $SVD$ even more evident, with $SVD$ showing the lowest $\mathcal{R}_{MR}$ value. Among all the methods, $SVD$ consistently demonstrates superior performance by obtaining the lowest $\mathcal{R}_{MR}$ and smaller $\mathcal{R}_{ME}$ across all matrix sizes. Overall, the proposed methods exhibit high efficiency, as they provide very low values for both $\mathcal{R}_{ME}$ and $\mathcal{R}_{MR}$ for computing the $MP$-inverse.
\end{example}

\section{Conclusion}\label{conc1.1}
This study introduced an effective strategy for solving an FLS by using $\mathcal{L}U$ decomposition, $QR$ decomposition, and $SVD$ for the block-structure matrix as a structural backbone. Using this framework, we developed a direct technique to compute the $\{1,2\}$-inverse, $\{1,2,3\}$-inverse and $MP$-inverse for obtaining the least squares solutions. Further, the criteria for necessary and sufficient conditions for consistency and the existence of a strong solution are discussed. The practicability of the approach was demonstrated through a complex analog circuit, and the Markov decision process involving fuzzy rewards demonstrates its ability to handle problems under uncertainty. While the $\{1, 2\}$-inverse, $\{1,2,3\}$-inverse and $MP$-inverse serve as a primary tool in this work, other generalized inverses, including group and Drazin inverses \cite{Ben} can further extend the scope of solving these singular systems. The following problems remain open for future work:
\begin{itemize}
    \item Exploring the proposed algorithm in fuzzy differential equations (FDE) of the form described in \cite{gap-1}, which frequently arise in applications such as electrical circuits and population dynamics.
    \item Investigating these direct techniques for FFLS which helps to capture and handle the broader uncertainty present in real-world applications.  
    \item Extending this framework to handle fuzzy linear programming problem as described in \cite{gap-2}, would be an interesting direction for providing effective optimization solutions under uncertainty.
\end{itemize}
\section{Declaration of competing interest} 
The authors declare that they do not have any kind of conflicting financial interest or personal relationship that might affect the results in this paper. 
\section{Data availability} This research paper does not involve the use of any data.

\end{document}